\documentclass[11pt, a4paper, oneside, reqno]{amsart}
\usepackage{graphicx}
\usepackage{float}
\usepackage{pstricks}
\usepackage{amscd}
\usepackage{amsmath}
\usepackage{amsxtra}
\usepackage{hyperref}
\usepackage[T1]{fontenc}
\usepackage[utf8]{inputenc}
\usepackage{lipsum}
\usepackage{tikz}
\usepackage{amssymb, amsthm}
\usepackage{url}
\usepackage{enumerate}
\usepackage{mathrsfs}
\usepackage{epsfig}
\usepackage[active]{srcltx}

\usetikzlibrary{arrows}
\usetikzlibrary{calc}
\theoremstyle{plain}
\newtheorem{theorem}{Theorem}[section]
\newtheorem{c-theorem}{Construction theorem}[section]

\newtheorem{lemma}[theorem]{Lemma}
\newtheorem{proposition}[theorem]{Proposition}
\newtheorem{corollary}[theorem]{Corollary}

\theoremstyle{definition}
\newtheorem{definition}[theorem]{Definition}

\theoremstyle{remark}
\newtheorem{remark}[theorem]{Remark}

\newcommand*{\Ge}{\geqslant}
\newcommand*{\Le}{\leqslant}

\makeatletter
\@namedef{subjclassname@2020}{%
\textup{2020} Mathematics Subject Classification}
\makeatother

\newcommand{\ncom}{\newcommand}
\ncom{\bq}{\begin{equation}}
\ncom{\eq}{\end{equation}}
\ncom{\beqn}{\begin{eqnarray*}}
\ncom{\eeqn}{\end{eqnarray*}}
\ncom{\beq}{\begin{eqnarray}}
\ncom{\eeq}{\end{eqnarray}}
\ncom{\nno}{\nonumber}
\ncom{\rar}{\rightarrow}
\ncom{\Rar}{\Rightarrow}
\ncom{\noin}{\noindent}
\ncom{\bc}{\begin{centre}}
\ncom{\ec}{\end{centre}}
\ncom{\sz}{\scriptsize}
\ncom{\rf}{\ref}
\ncom{\sgm}{\sigma}
\ncom{\Sgm}{\Sigma}
\ncom{\dt}{\delta}
\ncom{\Dt}{Delta}
\ncom{\lmd}{\lambda}
\ncom{\Lmd}{\Lambda}
\ncom{\eps}{\epsilon}
\ncom{\pcc}{\stackrel{P}{>}}
\ncom{\dist}{{\rm\,dist}}
\ncom{\im}{{\rm Im\,}}
\ncom{\sgn}{{\rm sgn\,}}
\ncom{\ba}{\begin{array}}
\ncom{\ea}{\end{array}}
\ncom{\eop}{\hfill{{\rule{2.5mm}{2.5mm}}}}
\ncom{\eof}{\hfill{{\rule{1.5mm}{1.5mm}}}}
\ncom{\hone}{\mbox{\hspace{1em}}}
\ncom{\htwo}{\mbox{\hspace{2em}}}
\ncom{\hthree}{\mbox{\hspace{3em}}}
\ncom{\hfour}{\mbox{\hspace{4em}}}
\ncom{\hsev}{\mbox{\hspace{7em}}}
\ncom{\vone}{\vskip 2ex}
\ncom{\vtwo}{\vskip 4ex}
\ncom{\vonee}{\vskip 1.5ex}
\ncom{\vthree}{\vskip 6ex}
\ncom{\vfour}{\vspace*{8ex}}
\ncom{\norm}{\|\;\;\|}
\ncom{\integ}[4]{\int_{#1}^{#2}\,{#3}\,d{#4}}
\ncom{\inp}[2]{\langle{#1},\,{#2} \rangle}
\ncom{\Inp}[2]{\Langle{#1},\,{#2} \Langle}
\ncom{\vspan}[1]{{{\rm\,span}\#1 \}}}
\ncom{\dm}[1]{\displaystyle {#1}}

\begin{document}
\title[Dirichlet-type spaces of the bidisc]{Dirichlet-type spaces of the bidisc \\ and Toral $2$-isometries}

\author[Santu Bera, Sameer Chavan and Soumitra Ghara]{Santu Bera, Sameer Chavan and Soumitra Ghara}

\address{Department of Mathematics and Statistics\\
Indian Institute of Technology Kanpur, 208016, India}

\email{santu20@iitk.ac.in}
\email{chavan@iitk.ac.in}  
\email{sghara@iitk.ac.in}

\thanks{The first author is supported through the PMRF Scheme (2301352), while the work of the third author is supported by INSPIRE Faculty Fellowship (DST/INSPIRE/04/2021/002555).}

\keywords{Dirichlet-type spaces, toral $2$-isometry, division property, Gleason's problem, 
Cowen-Douglas class, Koszul complex}

\subjclass[2020]{Primary 47A13, 32A36, 47B38; Secondary 31C25, 46E20}

\begin{abstract}  
We introduce and study Dirichlet-type spaces $\mathcal D(\mu_1, \mu_2)$ of the unit bidisc $\mathbb D^2,$ where $\mu_1, \mu_2$ are finite positive Borel measures on the unit circle. 
We show that the coordinate functions $z_1$ and $z_2$ are multipliers for $\mathcal D(\mu_1, \mu_2)$ and the complex polynomials are dense in $\mathcal D(\mu_1, \mu_2).$ Further, we obtain the division property and 
solve Gleason's problem for $\mathcal D(\mu_1, \mu_2)$ over a bidisc centered at the origin. In particular, we show that the commuting pair $\mathscr M_z$
of the multiplication operators $\mathscr M_{z_1},$ $\mathscr M_{z_2}$ on $\mathcal D(\mu_1, \mu_2)$ defines a cyclic toral $2$-isometry and $\mathscr M^*_z$ belongs to the Cowen-Douglas class ${\bf B}_1(\mathbb D^2_r)$ for some $r >0.$ Moreover, we formulate a notion of wandering subspace for commuting tuples and use it  
to obtain a bidisc analog of Richter's representation theorem for cyclic analytic $2$-isometries. 
In particular, we show that a cyclic analytic toral $2$-isometric pair $T$ with cyclic vector $f_0$ is unitarily equivalent to $\mathscr M_z$ on $\mathcal D(\mu_1, \mu_2)$ if and only if 
$\ker T^*,$ spanned by $f_0,$ is a 
wandering subspace for $T.$ 
\end{abstract}

\maketitle

\section{Introduction and preliminaries}

The aim of this paper is to obtain a bidisc counter-part of the theory of Dirichlet-type spaces of the open unit disc as presented in \cite{R1991} (see \cite{CGR2020} for a ball counter-part of this theory).  Throughout this paper, $\mathbb D$ denotes the open unit disc $\{z \in \mathbb C : |z|<1\}$ in the complex plane $\mathbb C.$ 
Recall that Dirichlet-type spaces of $\mathbb D$ are model spaces for the class of cyclic analytic $2$-isometries (see \cite{R1991}). Thus to arrive at an appropriate notion of the Dirichlet-type spaces of the unit bidisc $\mathbb D^2$, it is helpful to look for function spaces which support the class of $2$-isometries naturally associated with $\mathbb D^2.$ Let us first recall the definition of such $2$-isometries. 

For a complex Hilbert space $\mathcal H,$ let $\mathcal B(\mathcal H)$ denote the $C^*$-algebra of bounded linear operators on $\mathcal H.$
For a positive integer $d,$ a {\it commuting $d$-tuple $T$ on $\mathcal H$} is the $d$-tuple 
$(T_1, \ldots, T_d)$ of operators $T_1, \ldots, T_d \in \mathcal B(\mathcal H)$ satisfying $T_iT_j=T_jT_i,$ $1 \Le i \neq j \Le d.$ Let $T=(T_1, \ldots, T_d)$ be a commuting $d$-tuple on $\mathcal H$. We say that $T=(T_1, \ldots, T_d)$ is a {\it toral isometry} if $T_1, \ldots, T_d$ are isometries. 
Following \cite{A1990, AS1999, R1991}, $T$ is said to be a {\it toral $2$-isometry} if 
\beq
\label{C1}
I - T^*_iT_i - T^*_jT_j + T^*_jT^*_iT_iT_j=0, \quad i,j = 1, \ldots, d. 
\eeq
A toral isometry is necessarily a toral $2$-isometry, but the converse is not true (see \cite[Example~1]{AS1999}).  

To propose a successful analog of Dirichlet-type spaces on $\mathbb D^2,$ it is helpful to examine examples of toral $2$-isometries arising from function spaces. Since the operator of multiplication by the coordinate function on the classical Dirichlet space $\mathcal D(\mathbb D)$ is a $2$-isometry, it is natural to seek the classical Dirichlet space of the unit bidisc. 
Recall that the {\it Dirichlet space $\mathcal D(\mathbb D) \otimes \mathcal D(\mathbb D)$ of $\mathbb D^2$} 
is given by 
\beqn
\Big\{f \in \mathcal O(\mathbb D^2): \|f\|^2_{\mathcal D(\mathbb D) \otimes \mathcal D(\mathbb D)}:=\sum_{(m, n) \in \mathbb Z^2_+}|\hat{f}(m, n)|^2 (m+1)(n+1)< \infty\Big\},
\eeqn
where $\mathcal O(\Omega)$ denotes the space of holomorphic functions on a domain $\Omega,$  $\mathbb Z_+$ denotes the set of nonnegative integers and $\hat{f}$ denotes the Fourier transform of $f.$ 
It turns out that if $\mathscr M_{z_1}$ and $\mathscr M_{z_2}$ are the operators of multiplication by the coordinate functions $z_1$ and $z_2,$ respectively, on $\mathcal D(\mathbb D) \otimes \mathcal D(\mathbb D),$ then 
the commuting pair $(\mathscr M_{z_1}, \mathscr M_{z_2})$ satisfies \eqref{C1} for $1 \Le i=j \Le 2,$ but it fails to satisfy \eqref{C1} for $1 \Le i \neq j \Le 2.$ 
This failure may be attributed to the fact that the mapping $(m, n) \mapsto \|z^m_1z^n_2\|^2$ is a polynomial of bi-degree $(1, 1).$ Interestingly, there is a ``natural'' choice $\mathcal D(\mathbb D^2)$ of the Dirichlet space  containing $\mathcal D(\mathbb D) \otimes \mathcal D(\mathbb D)$ for which the associated pair $(\mathscr M_{z_1}, \mathscr M_{z_2})$ is a toral $2$-isometry: 
$$\mathcal D(\mathbb D^2)=\Big\{f \in \mathcal O(\mathbb D^2): \|f\|^2_{\mathcal D(\mathbb D^2)}:=\sum_{(m, n) \in \mathbb Z^2_+}|\hat{f}(m, n)|^2 (m+n+1) < \infty\Big\}.$$
The norm $\|\cdot\|_{\mathcal D(\mathbb D^2)}$ can also be written as follows:
\beq \label{Arcozzi-formula}
\|f\|^2_{\mathcal D(\mathbb D^2)} &=& \|f\|^2_{H^2(\mathbb D^2)} + \sup_{0 < r < 1} \int_{\mathbb T} \int_{\mathbb D}|\partial_1 f(z_1, re^{i \theta})|^2  \,dA(z_1)d\theta \notag \\
&+& \sup_{0 < r < 1} \int_{\mathbb T} \int_{\mathbb D}|\partial_2 f(re^{i \theta}, z_2)|^2  \,dA(z_2)d\theta, 
\eeq 
where $d\theta$ (resp.  
$dA$) denotes the \emph{normalized Lebesgue arc-length $($resp. area$)$ measure} on $\mathbb T$ (resp. $\mathbb D$).
Recall that the {\it Hardy space} $H^2(\mathbb D^d)$ of the unit polydisc $\mathbb D^d$ is the reproducing kernel Hilbert space (see \cite{PR2016} for the definition of the reproducing kernel Hilbert space) associated with the {\it Cauchy kernel} 
\beqn
\kappa(z, w)=\prod_{j=1}^d(1-z_j\overline{w}_j)^{-1}, \quad z=(z_1, \ldots, z_d), ~w=(w_1, \ldots, w_d) \in \mathbb D^d,
\eeqn 
where $d$ is a positive integer. 
It is worth noting that for any $f \in H^2(\mathbb D^d),$ 
\beq \label{Hardy-norm-new-0}
\|f\|^2_{H^2(\mathbb D^d)} &=& \sum_{\alpha \in \mathbb Z^d_+}|\hat{f}(\alpha)|^2  \\ \label{Hardy-norm-new}
&=& \sup_{0 < r < 1}\int_{[0, 2\pi]^d}|f(re^{i \theta_1}, \ldots, re^{i \theta_d})|^2d\theta_1 \cdots d \theta_d
\eeq
(see \cite[Section~3.4]{Ru1969}).


For a nonempty subset $\Omega$ of $\mathbb C^d,$ let 
$M_+(\Omega)$ denote the set of finite positive Borel measures on $\Omega.$ Let 
$P_{\mu}(w)$ denote the \emph{Poisson integral} $\int_{\mathbb T} \frac{1-|w|^2}{|w-\zeta|^2}d\mu(\zeta)$ of the measure $\mu \in M_+(\mathbb T).$ 
For future reference, we record the following consequence of the Fubini theorem (see \cite[Theorem~8.8]{Ru1987}) and the fact that the mapping $r \mapsto \int_{\mathbb T} |f(z, re^{i \theta})|^2 d\theta$ is increasing.  
\begin{lemma} \label{increasing-lem}
For $f \in \mathcal O(\mathbb D^{d+1})$ and $\mu \in M_+(\mathbb D^d)$, the extended real-valued mapping  
$\phi(r)= \int_{\mathbb T} \int_{\mathbb D^d}|f(z, re^{i \theta})|^2 d\mu(z)d\theta,$ $r \in (0, 1),$ is increasing.  
\end{lemma}

The formula \eqref{Arcozzi-formula} together with Richter's notion of Dirichlet-type spaces (see \cite[Sect.~3]{R1991}) motivates us to the following:

\begin{definition} 
For $\mu_1, \mu_2 \in M_+(\mathbb T)$ and $f \in \mathcal O(\mathbb D^2),$ {\it the Dirichlet integral} $D_{\mu_1, \mu_2}(f)$ of $f$ is given by
\beqn
D_{\mu_1, \mu_2}(f) &=& \sup_{0 < r < 1} \int_{\mathbb T} \int_{\mathbb D}|\partial_1 f(z_1, re^{i \theta})|^2 P_{\mu_1}(z_1) \,dA(z_1)d\theta \\
&+& \sup_{0 < r < 1} \int_{\mathbb T} \int_{\mathbb D}|\partial_2 f(re^{i \theta}, z_2)|^2 P_{\mu_2}(z_2) \,dA(z_2)d\theta.
\eeqn 
If either $\mu_1$ or $\mu_2$ is $0,$ then the {\it Dirichlet-type space} $\mathcal D(\mu_1, \mu_2)$ is the space of  functions $f \in H^2(\mathbb D^2)$ satisfying $D_{\mu_1, \mu_2}(f) < \infty.$ Otherwise, we set $\mathcal D(\mu_1, \mu_2) = \{f \in \mathcal O(\mathbb D^2) : D_{\mu_1, \mu_2}(f) < \infty\}.$
\end{definition}

Before we define a norm on the Dirichlet-type space $\mathcal D(\mu_1, \mu_2),$ we present a $2$-variable analog of \cite[Lemma 3.1]{R1991}. 
\begin{lemma} \label{Rmk-Hardy-inclusion}
For $\mu_1, \mu_2 \in M_+(\mathbb T),$
$\mathcal D(\mu_1, \mu_2) \subseteq H^2(\mathbb D^2).$ 
\end{lemma} 
\begin{proof} By the definition of $\mathcal D(\mu_1, \mu_2),$ we may assume that both measures $\mu_1$ and $\mu_2$ are non-zero. 
Note that 
\beq \label{Poisson-lb} 
P_{\mu}(w) \Ge \frac{\mu(\mathbb T)}{4}(1-|w|^2), \quad \mu \in M_+(\mathbb T), ~w \in \mathbb D. 
\eeq
Thus, for any $f(z_1, z_2)=\sum_{m, n =0}^{\infty}a_{m, n}\,z^m_1z^n_2 \in \mathcal D(\mu_1, \mu_2),$ 
\beqn
&& \int_{\mathbb T} \int_{\mathbb D}|\partial_1 f(z_1, re^{i \theta})|^2 P_{\mu_1}(z_1) \,dA(z_1)d\theta \\
& \overset{\eqref{Poisson-lb} }\Ge & \frac{\mu_1(\mathbb T)}{4} \sum_{m =1}^{\infty} \sum_{n=0}^{\infty} |a_{m, n}|^2 m^2 r^{2n} \int_{\mathbb D} |z^{m-1}_1|^2 (1-|z_1|^2)dA(z_1) \\
&=& \frac{\mu_1(\mathbb T)}{4} \sum_{m =1}^{\infty} \sum_{n=0}^{\infty} |a_{m, n}|^2 \frac{m r^{2n}}{m+1}. 
\eeqn
A similar estimate using \eqref{Poisson-lb} gives
\beqn
&& \int_{\mathbb T} \int_{\mathbb D}|\partial_2 f(re^{i \theta}, z_2)|^2 P_{\mu_2}(z_2) \,dA(z_2)d\theta \Ge \frac{\mu_2(\mathbb T)}{4} \sum_{m =0}^{\infty} \sum_{n=1}^{\infty} |a_{m, n}|^2 \frac{nr^{2m}}{n+1}. 
\eeqn
Since $f \in \mathcal D(\mu_1, \mu_2),$
\beqn
 \sup_{0 < r < 1} \sum_{m =1}^{\infty} \sum_{n=0}^{\infty} |a_{m, n}|^2 \frac{m r^{2n}}{m+1} < \infty, \quad \sup_{0 < r < 1}  \sum_{m =0}^{\infty} \sum_{n=1}^{\infty} |a_{m, n}|^2 \frac{nr^{2m}}{n+1} < \infty.
 \eeqn
It is now easy to see using the monotone convergence theorem (see \cite[Theorem~1.26]{Ru1987}) that 
$f$ belongs to $H^2(\mathbb D^2).$ 
\end{proof}

In view of Lemmas \ref{increasing-lem} and \ref{Rmk-Hardy-inclusion}, the Dirichlet-type space $\mathcal D(\mu_1, \mu_2)$ can be endowed with the norm
\beqn
\|f\|^2_{\mathcal D(\mu_1, \mu_2)} &=& \|f\|^2_{H^2(\mathbb D^2)} + \lim_{r \rar 1^{-}} \int_{\mathbb T} \int_{\mathbb D}|\partial_1 f(z_1, re^{i \theta})|^2 P_{\mu_1}(z_1) \,dA(z_1)d\theta \\
&+& \lim_{r \rar 1^{-}} \int_{\mathbb T} \int_{\mathbb D}|\partial_2 f(re^{i \theta}, z_2)|^2 P_{\mu_2}(z_2) \,dA(z_2)d\theta. 
\eeqn
We see that $\mathcal D(\mu_1, \mu_2)$ is a reproducing kernel Hilbert space (see Lemma~\ref{r-kernel}).

The present paper is devoted to the study of Dirichlet-type spaces with efforts to understand the bidisc counter-part of the work carried out in \cite{R1991}.  
Before we state the main results of this paper, we need some definitions. 

Let $\Omega$ be a domain in $\mathbb C^d$ and 
let $\mathscr H$ be a Hilbert space such that $\mathscr H \subseteq \mathcal O(\Omega).$ 
A function $\varphi : \Omega \rar \mathbb C$ is said to be a {\it multiplier} of $\mathscr H$ if $\varphi f \in \mathscr H$ for every $f \in \mathscr H.$
For a nonempty subset $U$ of $\Omega,$ we say that {\it Gleason's problem can be solved for $\mathscr H$ over $U$} if for every $f \in \mathscr H$ and $\lambda \in U,$ there exist functions $g_1, \ldots, g_d$ in $\mathscr H$ such that
\beqn
f(z)=f(\lambda) + \sum_{j=1}^d (z_j-\lambda_j)g_j(z), \quad z=(z_1, \ldots, z_d) \in \Omega.
\eeqn
We say that {\it Gleason's problem can be solved for $\mathscr H$} if Gleason's problem can be solved for $\mathscr H$ over $\Omega$ (the reader is referred to \cite{Z1988} for a solution of Gleason's problem for Bergman and Bloch spaces of the unit ball). It turns out that Gleason's problem can be solved for $H^2(\mathbb D^d)$ (see Remark~\ref{rmk-abstract-G}).

\begin{definition} Let $\Omega$ be a domain in $\mathbb C^d$ and 
let $\mathscr H$ be a Hilbert space such that $\mathscr H \subseteq \mathcal O(\Omega).$ 
We say that 
$\mathscr H$ has 
the {\it $j$-division property}, $j=1, \ldots, d,$ if $\frac{f(z)}{z_j-\lambda_j}$ defines a function in $\mathscr H$  
whenever $\lambda \in \Omega,$ $f \in \mathscr H$ and  $\{z \in \Omega : z_j =\lambda_j\}$ is contained in $Z(f),$ the zero set of $f.$ If $\mathscr H$ has $j$-division property for every $j=1, \ldots, d,$ then we say that $\mathscr H$ has 
the {\it division property}.
\end{definition}
In case of $d=1,$ this property appeared in \cite[Definition~1.1]{AR1997}. 
One of the main results of this paper shows that $\mathcal D(\mu_1, \mu_2)$ has the division property. 
In what follows, we require a generalization of the notion of the wandering subspace introduced by Halmos (see \cite[P.~103]{H1961}).
\begin{definition} \label{def-ws}
Let $T=(T_1, \ldots, T_d)$ be a commuting $d$-tuple on $\mathcal H.$ A closed subspace $\mathcal W$ of $\mathcal H$ is said to be {\it wandering} for $T$ if for every $i=1, \ldots, d,$ 
\beqn
\prod_{{j=1}}^d T^{\alpha_j}_j \mathcal W \perp \prod_{j=1}^d T^{\beta_j}_j \mathcal W, \quad \alpha_j, \beta_j \in \mathbb Z_+, ~j=1, \ldots, d, ~\alpha_i =0, ~\beta_i \neq 0.
\eeqn 
\end{definition}
\begin{remark} \label{rmk-wandering}
If $d=1,$ then $\mathcal W$ is a wandering subspace for $T$ if and only if $\mathcal W \perp T^k(\mathcal W)$ for every integer $k \Ge 1.$  In particular, $\ker T^*$ is a wandering subspace for any $T \in \mathcal B(\mathcal H).$ Moreover, if $T=(T_1, \ldots, T_d)$ is a commuting $d$-tuple such that $ T^*_{j}T_{i}=T_{i}T^*_{j},$ $1 \Le i \neq j \Le d,$ then $\ker T^*$ is a wandering subspace for $T.$ 
\end{remark} 
It follows from Remark~\ref{rmk-wandering} that the space spanned by the constant function $1$ is a wandering subspace for the multiplication $2$-tuple $\mathscr M_z$ on $H^2(\mathbb D^2).$ Interestingly, this fact extends to the multiplication $2$-tuple $\mathscr M_z$ on $\mathcal D(\mu_1, \mu_2)$ (see Corollary~\ref{wandering-s}).   


Recall that a commuting $d$-tuple $T=(T_1, \ldots, T_d)$ on $\mathcal H$ is {\it cyclic} with {\it cyclic vector} $f_0 \in \mathcal H$ if $\bigvee \big\{T^{\alpha}f_0 :  \alpha \in \mathbb Z^d_+\big\}=\mathcal H,$ where $\bigvee$ denotes the closed linear span and $T^{\alpha} =\prod_{j=1}^d T^{\alpha_j}_j.$ 
For later purpose, we state the following property of cyclic tuples (see \cite[Proposition~1.1]{AW1990}):
\beq \label{cyclic-kernel}
\mbox{\it If $T$ is cyclic, then for any $\omega \in \mathbb C^d,$  $\dim \ker(T^* - \omega)$ is at most $1,$} 
\eeq
where $\ker S=\cap_{j=1}^d \ker S_j$ for the $d$-tuple $S=(S_1, \ldots, S_d)$  
and $\dim$ stands for the Hilbert space dimension.
A commuting $d$-tuple $T$ on $\mathcal H$ has the {\it wandering subspace property} if
 $\mathcal H=\bigvee
_{\alpha \in \mathbb Z_+} T^\alpha(\ker T^*).$
Following \cite[P.~56]{EL2018}, we say that a commuting $d$-tuple $T=(T_1, \ldots, T_d)$ on $\mathcal H$ is {\it analytic} if 
$$\bigcap_{k=0}^{\infty} \sum_{\alpha \in \Gamma_k}T^{\alpha}\mathcal H = \{0\},$$ where, for $k \in \mathbb Z_+,$ $\Gamma_k := \{\alpha =(\alpha_1, \ldots, \alpha_d) \in \mathbb Z^d_+: \alpha_1 + \cdots + \alpha_d = k\}.$ 
Note that if $T$ is analytic, then $T_1, \ldots, T_d$ are analytic. 

Let $\Omega$ be a domain in $\mathbb C^d.$
  For a positive integer $n$, let ${\bf B}_n(\Omega)$ denote the set
  of all commuting $d$-tuples $T$ on $\mathcal H$ satisfying the following \index{${\bf B}_n(\Omega)$}
  conditions:
  \begin{enumerate}
  \item[$\bullet$]  for every $\omega=(\omega_1, \ldots, \omega_d) \in \Omega$, the map $D_{T- \omega}(x) = ((T_j-\omega_j)x)_{j=1}^d$ from $\mathcal H$ into $\mathcal H^{\oplus d}$ has closed range and $\dim{\ker({T-\omega})} = n,$
     \item[$\bullet$] the subspace
    $\bigvee_{\omega \in \Omega}  \ker({T-\omega})$ of $\mathcal H$ equals $\mathcal H$.
  \end{enumerate}
  We call the set ${\bf B}_n(\Omega)$ the {\it Cowen-Douglas class
    of rank $n$ with respect to $\Omega$} (refer to \cite{CD1978, CuS1984} for the basic theory of Cowen-Douglas class).

\section{Statements of main theorems}

The following three theorems collect several basic properties of Dirichlet-type spaces $\mathcal D(\mu_1, \mu_2).$ 
\begin{theorem} \label{thm1} For $\mu_1, \mu_2 \in M_+(\mathbb T),$ we have the following statements$:$
\begin{enumerate}
\item[$\mathrm{(i)}$] the coordinate functions $z_1, z_2$ are multipliers of $\mathcal D(\mu_1, \mu_2),$
\item[$\mathrm{(ii)}$] the polynomials are dense in $\mathcal D(\mu_1, \mu_2),$
\item[$\mathrm{(iii)}$] for non-negative integers $k, l$  and a polynomial $p$ in $z_1$ and $z_2,$ 
\beq \notag 
 \|z^k_1z^l_2p\|^2_{\mathcal D(\mu_1, \mu_2)} 
&=& \|p\|^2_{\mathcal D(\mu_1, \mu_2)} + k \int_{\mathbb T^2} |p(e^{i \eta},  e^{i \theta})|^2 d\mu_1(\eta)d\theta \\
&+& l \int_{\mathbb T^2} |p(e^{i \theta}, e^{i \eta})|^2 d\mu_2(\eta)d\theta. \label{formula-Richter}
\eeq 
\end{enumerate}
\end{theorem}


\begin{theorem}
\label{D-mu-estimate}
For $\mu_1, \mu_2 \in M_+(\mathbb T),$ 
$\mathcal D(\mu_1, \mu_2)$ has the division property.
\end{theorem}




\begin{theorem} \label{Gleason-new} For $\mu_1, \mu_2 \in M_+(\mathbb T),$ Gleason's problem can be solved for $\mathcal D(\mu_1, \mu_2)$ over $\mathbb D^2_r$ for some $r \in (0, 1).$ 
\end{theorem}

Here $\mathbb D^2_r$ denotes the bidisc $\{(z_1, z_2) \in \mathbb C^2 : |z_1|< r, |z_2|< r\},$ where $r$ is a positive real number. Unlike the one variable situation, we do not know whether Gleason's problem can be solved for $\mathcal D(\mu_1, \mu_2)$ over the unit bidisc. 
It is worth noting that not all facts about Dirichlet-type spaces of the unit disc have successful counterparts in the bidisc case. For example, the commuting pair $\mathscr M_z=(\mathscr M_{z_1}, \mathscr M_{z_2})$ on $\mathcal D(\mu_1, \mu_2)$ fails to be essentially normal (see Corollary~\ref{coro-e-n}). Moreover, the verbatim analog of the model theorem \cite[Theorem~5.1]{R1991} does not hold true (see Remark~\ref{model-rmk}). 

The following result asserts that $\mathscr M_z$ on $\mathcal D(\mu_1, \mu_2)$ is a canonical model for 
analytic $2$-isometries for which $\ker T^*$ is a cyclic wandering subspace.
 
\begin{theorem}[A representation theorem] \label{model-thm}
Let $T=(T_1, T_2)$ be a commuting pair on $\mathcal H.$ Then the following statements are equivalent:
\begin{enumerate}[\rm(i)]
\item $T$ is a cyclic analytic toral $2$-isometry with cyclic vector $f_0 \in \ker T^*$ and $\ker T^*$ is a wandering subspace for $T,$
\item $T$ is a cyclic toral $2$-isometry with cyclic vector $f_0 \in \ker T^*$, $T^*$ belongs to ${\bf B}_1(\mathbb D^2_r)$ for some $r \in (0, 1)$ and $\ker T^*$ is a wandering subspace for~$T,$
\item there exist $\mu_1, \mu_2 \in M_+(\mathbb T)$ such that $T$ is unitarily equivalent to $\mathscr M_z$ on $\mathcal D(\mu_1, \mu_2).$
\end{enumerate}
\end{theorem}
\begin{remark} \label{model-rmk}
By \cite[Theorem~1]{R1988}, any analytic $2$-isometry $T$ on $\mathcal H$ has the wandering subspace property. This result fails even for analytic toral isometric $d$-tuples if $d > 1.$ Indeed, if $a \in \mathbb D^2 \backslash \{(0, 0)\},$ then 
the restriction of $\mathscr M_z$ to $\{f \in H^2(\mathbb D^2) : f(a)=0\}$ is a toral isometry without the wandering subspace property. 
This may be seen by imitating the argument of \cite[Example~6.8]{BEKS2017} with the only change that the application of \cite[Theorem~4.3]{GRS2005} is replaced by that of \cite[Corollary~4.6]{GRS2005}. This example also shows that the assumption that the cyclic vector $f_0$ belongs to $\ker T^*$ in (i) can not be dropped from Theorem~\ref{model-thm}. Also, by Theorem~\ref{thm1}(ii), the cyclicity of $T$ in (ii) of Theorem~\ref{model-thm} can not be relaxed.  
\end{remark}

Theorems~\ref{thm1}, \ref{Gleason-new} and \ref{model-thm} provide bidisc analogs of \cite[Theorems~3.6, 3.7 and 5.1]{R1991}, respectively. 
Also, Theorem~\ref{D-mu-estimate} presents a counterpart of the fact that Dirichlet-type spaces on the unit disc have the division property (see \cite[Corollary~3.8]{R1991} and \cite[Lemma~2.1]{R1987}). 
The proofs of these results and their consequences are presented in Sections~\ref{S2}-\ref{S4} (see Corollaries~\ref{cyclic-coro}, \ref{T-spectrum}, \ref{wandering-s}, \ref{coro-e-n}, \ref{exact-middle}, \ref{Cowen-Douglas}, \ref{description-kernel}, \ref{model-coro}, \ref{toral-iso-coro}). In the final short section, we discuss the spectral picture of the multiplication $2$-tuple $\mathscr M_z$ on $\mathcal D(\mu_1, \mu_2)$ and raise some related questions. 

\section{Proof of Theorem~\ref{thm1} and its consequences} \label{S2}

We need several lemmas to prove Theorem~\ref{thm1}.  
\begin{lemma} \label{r-kernel}
The Dirichlet-type space $\mathcal D(\mu_1, \mu_2)$ is a reproducing kernel Hilbert space. If $\kappa : \mathbb D \times \mathbb D \rar \mathbb C$ is the reproducing kernel of $\mathcal D(\mu_1, \mu_2),$ then for any $r \in (0, 1),$ 
$\bigvee \{\kappa(\cdot, w) : |w| < r\}=\mathcal D(\mu_1, \mu_2)$ and 
$\kappa(\cdot, 0)=1.$ 
\end{lemma}
\begin{proof}
We borrow an argument from the proof of \cite[Theorem~1.6.3]{EKMR2014}. 
Let $\{f_n\}_{n \Ge 0}$ be a Cauchy sequence in $\mathcal D(\mu_1, \mu_2).$ Since $H^2(\mathbb D^2)$ is complete (see \cite[p~53]{Ru1969}), there exists a $f \in H^2(\mathbb D^2)$ such that $\|f_n-f\|^2_{H^2(\mathbb D^2)} \rar 0$ as $n \rar \infty.$ Moreover, since $H^2(\mathbb D^2)$ is a reproducing kernel Hilbert space, for every $j=1, 2,$ $\partial_j f_n$ converges compactly to $\partial_j f$ on $\mathbb D^2.$ 
Also, since $\{f_n\}_{n \Ge 0}$ is bounded in $\mathcal D(\mu_1, \mu_2)$, by Lemma~\ref{increasing-lem}, 
there exists an $M > 0$ such that for every integer $n \Ge 0$ and $r \in (0, 1),$ 
\beqn
&&  \int_{\mathbb T} \int_{\mathbb D}|\partial_1 f_n(z_1, re^{i \theta})|^2 P_{\mu_1}(z_1) \,dA(z_1)d\theta < M, \\
&& \int_{\mathbb T} \int_{\mathbb D}|\partial_2 f_n(re^{i \theta}, z_2)|^2 P_{\mu_2}(z_2) \,dA(z_2)d\theta < M.
\eeqn
By Fatou's lemma (see \cite[Lemma~1.28]{Ru1987}), for any $r \in (0, 1),$ 
\allowdisplaybreaks
\beq \label{fn-fm}
&& \int_{\mathbb T} \int_{\mathbb D}|\partial_1 f(z_1, re^{i \theta})|^2 P_{\mu_1}(z_1) \,dA(z_1)d\theta \\
&\Le & \liminf_n \int_{\mathbb T} \int_{\mathbb D}|\partial_1 f_n(z_1, re^{i \theta})|^2 P_{\mu_1}(z_1) \,dA(z_1)d\theta \Le M. \notag
\eeq
Similarly, one can see that 
\beqn
\int_{\mathbb T} \int_{\mathbb D}|\partial_2 f(re^{i \theta}, z_2)|^2 P_{\mu_2}(z_2) \,dA(z_2)d\theta \Le M, \quad r \in (0, 1). 
\eeqn
This shows that $f \in \mathcal D(\mu_1, \mu_2).$ 
We may now argue as in \eqref{fn-fm} (with $f$ replaced by $f_n-f$ and $f_n$ replaced by $f_n-f_m$) and use Fatou's lemma to conclude that 
\beqn
&& \int_{\mathbb T} \int_{\mathbb D}|\partial_1 (f_n-f)(z_1, re^{i \theta})|^2 P_{\mu_1}(z_1) \,dA(z_1)d\theta \\
&\Le & \liminf_m \int_{\mathbb T} \int_{\mathbb D}|\partial_1 (f_n-f_m)(z_1, re^{i \theta})|^2 P_{\mu_1}(z_1) \,dA(z_1)d\theta.
\eeqn
Similarly, we obtain 
\beqn
&& \int_{\mathbb T} \int_{\mathbb D}|\partial_2 (f_n-f)(re^{i \theta}, z_2)|^2 P_{\mu_2}(z_2) \,dA(z_2)d\theta \\
&\Le & \liminf_m \int_{\mathbb T} \int_{\mathbb D}|\partial_2 (f_n-f_m)(re^{i \theta}, z_2)|^2 P_{\mu_2}(z_2) \,dA(z_2)d\theta.
\eeqn
These two estimates combined with Lemma~\ref{increasing-lem} yield
\beqn
D_{\mu_1, \mu_2}(f_n - f) \Le \liminf_{m} D_{\mu_1, \mu_2}(f_n - f_m), \quad n \Ge 0. 
\eeqn
This shows that $\{f_n\}_{n \Ge 0}$ converges to $f$ in $\mathcal D(\mu_1, \mu_2).$ Finally, since $H^2(\mathbb D^2)$ is a reproducing kernel Hilbert space, so is $\mathcal D(\mu_1, \mu_2)$ (see Lemma~\ref{Rmk-Hardy-inclusion}). 

To see the `moreover' part, note that for any $f \in \mathcal D(\mu_1, \mu_2),$ by the reproducing property of $\mathcal D(\mu_1, \mu_2),$
\beqn
\inp{f}{1}_{\mathcal D(\mu_1, \mu_2)} = \inp{f}{1}_{H^2(\mathbb D^2)} = f(0)=\inp{f}{\kappa(\cdot, 0)}_{\mathcal D(\mu_1, \mu_2)}, 
\eeqn
and hence $\kappa(\cdot, 0)=1.$ The rest follows from the reproducing property of $\mathcal D(\mu_1, \mu_2)$ together with an application of the identity theorem. 
\end{proof}

Although we do not need in this section the full strength of the following lemma (cf. \cite[Theorem 4.2]{GR2006}), we include it for later usage: 
\begin{lemma} \label{Hardy-slice}
Let $f : \mathbb D^2 \rar \mathbb C$ be a holomorphic function. 
For $r \in (0, 1)$ and $\theta \in [0, 2\pi],$ consider the holomorphic function $f_{r, \theta}(w)=f(w, re^{i \theta}),$ $w \in \mathbb D.$ 
If $f \in H^2(\mathbb D^2),$ then $f_{r, \theta} \in H^2(\mathbb D)$ for every $r \in (0, 1)$ and $\theta \in [0, 2\pi].$ Moreover,
\beq \label{Hardy-fact}
\sup_{0 < r < 1}\int_{0}^{2\pi}\|f_{r, \theta}\|^2_{H^2(\mathbb D)}d\theta =\|f\|^2_{H^2(\mathbb D^2)}, \quad f \in H^2(\mathbb D^2).
\eeq 
\end{lemma}
\begin{proof} The proof relies on the formula \eqref{Hardy-norm-new-0}. 
First, note that 
\beq \label{relation-norm}
\sum_{m =0}^{\infty} |\hat{f}_{r, \theta}(m)|^2 = \sum_{m=0}^{\infty} \Big|\sum_{n=0}^{\infty} \hat{f}(m, n) r^n e^{in \theta}\Big|^2.  
\eeq
If $f \in H^2(\mathbb D^2),$ then applying the Cauchy-Schwarz inequality to \eqref{relation-norm} gives that $f_{r, \theta} \in H^2(\mathbb D)$ for every $r \in (0, 1)$ and $\theta \in [0, 2\pi].$ Moreover, integrating both sides of \eqref{relation-norm} with respect to $\theta$ over $[0, 2\pi]$ yields \eqref{Hardy-fact}. 
\end{proof}
\begin{remark}
We note that 
\beq
\notag
&&\mbox{if $f \in \mathcal D(\mu_1, \mu_2),$ then for every $r \in (0, 1),$ 
$f(\cdot, re^{i\theta}) \in \mathcal D(\mu_1)$} \\  \label{rmk-slice-theta}
&&\mbox{and $f(re^{i\theta}, \cdot) \in \mathcal D(\mu_2)$ for almost every $\theta \in [0, 2\pi].$} 
\eeq
To see this, note that for any holomorphic function $f : \mathbb D^2 \rar \mathbb C,$ 
\beq \label{formula-D}
D_{\mu_1, \mu_2}(f) = \lim_{r \rar 1^{-}} \int_{\mathbb T} D_{\mu_1}(f(\cdot, re^{i\theta})) d\theta 
+ \lim_{r \rar 1^{-}} \int_{\mathbb T} D_{\mu_2}(f(re^{i\theta}, \cdot))d\theta,
\eeq
and hence, if $f \in \mathcal D(\mu_1, \mu_2),$ then 
by Lemma~\ref{increasing-lem}, $\int_{\mathbb T} D_{\mu_1}(f(\cdot, re^{i\theta})) d\theta$ and $\int_{\mathbb T} D_{\mu_2}(f(re^{i\theta}, \cdot))d\theta$ are finite for every $r \in (0, 1).$ One may now apply Lemma~\ref{Hardy-slice} to complete the verification of \eqref{rmk-slice-theta}.  
\end{remark}
It turns out that the operator $\mathscr M_{z_j}$ of multiplication by the coordinate functions $z_j,$ $j=1, 2,$ defines a bounded linear operator on $\mathcal D(\mu_1, \mu_2).$
\begin{lemma} \label{multiplier}
The coordinate functions $z_1, z_2$ are multipliers of $\mathcal D(\mu_1, \mu_2).$
\end{lemma}
\begin{proof} 
By \eqref{rmk-slice-theta}, for any $f \in \mathcal D(\mu_1, \mu_2)$ and $r \in (0, 1),$ $f(\cdot, re^{i \theta}) \in \mathcal D(\mu_1)$ for a.e.\,$\theta \in[0,2\pi].$
By \cite[Theorem~3.6]{R1991},  the operator $\mathscr M_w$ of multiplication by the coordinate function $w$ on $\mathcal D(\mu_1)$ is bounded and satisfies 
\beq \label{bdd-mu1}
\|\mathscr M_w f(\cdot, re^{i \theta})\|_{\mathcal D(\mu_1)} \Le \|\mathscr M_w\| \|f(\cdot, re^{i \theta})\|_{\mathcal D(\mu_1)}~\mbox{for~a.e.}\,\theta \in[0,2\pi].
\eeq
Since $\mathscr M^*_w\mathscr M_w \Ge I,$ $\|\mathscr M_w\| \Ge 1.$
Fix now $f \in \mathcal D(\mu_1, \mu_2).$ By Lemma~\ref{Rmk-Hardy-inclusion}, $f \in H^2(\mathbb D^2),$ and hence $z_1f \in H^2(\mathbb D^2).$  
By \eqref{formula-D} (two applications), 
\beqn
&& D_{\mu_1, \mu_2}(z_1f) \\
&=& 
 \lim_{r \rar 1^{-}} \int_{\mathbb T} D_{\mu_1}((z_1f)(\cdot, re^{i\theta}))d\theta 
+ \lim_{r \rar 1^{-}} \int_{\mathbb T} D_{\mu_2}((z_1f)(re^{i\theta}, \cdot))d\theta \\
&\Le &  \lim_{r \rar 1^{-}} \int_{\mathbb T}\|\mathscr M_w f(\cdot, re^{i \theta})\|^2_{\mathcal D(\mu_1)} d\theta 
+ \lim_{r \rar 1^{-}} \int_{\mathbb T} D_{\mu_2}(f(re^{i\theta}, \cdot))d\theta \\
& \overset{\eqref{bdd-mu1}}\Le & \|\mathscr M_w\|^2 \lim_{r \rar 1^{-}} \int_{\mathbb T}\|f(\cdot, re^{i \theta})\|^2_{\mathcal D(\mu_1)} d\theta 
+ \lim_{r \rar 1^{-}} \int_{\mathbb T} D_{\mu_2}(f(re^{i\theta}, \cdot))d\theta \\
&\Le & \|\mathscr M_w\|^2 \lim_{r \rar 1^{-}} \int_{\mathbb T}\|f(\cdot, re^{i \theta})\|^2_{H^2(\mathbb D)} d\theta + \|\mathscr M_w\|^2  D_{\mu_1, \mu_2}(f)\\
& \overset{\eqref{Hardy-fact}}= & \|\mathscr M_w\|^2 \|f\|^2_{\mathcal D(\mu_1, \mu_2)}.
\eeqn
Similarly, one can see that for some $c_2 \Ge 1,$  
\beqn
D_{\mu_1, \mu_2}(z_2f) \Le c_2 \|f\|^2_{\mathcal D(\mu_1, \mu_2)}, \quad f \in  \mathcal D(\mu_1, \mu_2).
\eeqn
This completes the proof. 
\end{proof}

The following is a bidisc-analog of Richter's formula (see \cite[Proof of Theorem 4.1]{R1991}, \cite[Theorem 1.3]{CGR2020}). 
\begin{lemma} \label{Richter-formula}
For nonnegative integers $k, l$ and polynomial $p$ in the complex variables $z_1$ and $z_2,$ we have the formula \eqref{formula-Richter}.
\end{lemma}
\begin{proof}
By \eqref{rmk-slice-theta} (see also \eqref{formula-D}) and \cite[Proof of Theorem 4.1]{R1991}, 
\allowdisplaybreaks
\beqn
&& \|z^k_1z^l_2p\|^2_{\mathcal D(\mu_1, \mu_2)} \\
&= & \|z^k_1z^l_2 p\|^2_{H^2(\mathbb D^2)} +\lim_{r \rar 1^{-}} \int_{\mathbb T} 
  D_{\mu_1}( w^k p(w, re^{i \theta}))d\theta 
\\ &+& \lim_{r \rar 1^{-}} \int_{\mathbb T} D_{\mu_2}( w^l p(re^{i \theta}, w))d\theta \\
&=& \|p\|^2_{H^2(\mathbb D^2)} + \lim_{r \rar 1^{-}} \int_{\mathbb T} 
  \Big(D_{\mu_1}(p(w, re^{i \theta}))+ k \int_{\mathbb T} |p(e^{i \eta}, r e^{i \theta})|^2 d\mu_1(\eta)\Big)d\theta 
\\ &+& \lim_{r \rar 1^{-}} \int_{\mathbb T} \Big(D_{\mu_2}(p(re^{i \theta}, w))+ l \int_{\mathbb T} |p(r e^{i \theta}, e^{i \eta})|^2 d\mu_2(\eta)\Big)d\theta \\
&=& \|p\|^2_{\mathcal D(\mu_1, \mu_2)} + k \int_{\mathbb T^2} |p(e^{i \eta},  e^{i \theta})|^2 d\mu_1(\eta)d\theta + l \int_{\mathbb T^2} |p(e^{i \theta}, e^{i \eta})|^2 d\mu_2(\eta)d\theta, 
\eeqn
where we used Lemma~\ref{increasing-lem} and the monotone convergence theorem. 
\end{proof}

For $R=(R_1, R_2) \in (0, 1)^2$ and $f \in \mathcal O(\mathbb D^2),$ let $f_R(z)=f(R_1z_1, R_2z_2).$ 
To get the polynomial density in $\mathcal D(\mu_1, \mu_2),$ we need the following inequality. 
\begin{lemma} \label{contractivity-D}
For any $R=(R_1, R_2) \in (0, 1)^2$ and $f \in \mathcal D(\mu_1, \mu_2),$  
$$D_{\mu_1, \mu_2}(f_R) \Le D_{\mu_1, \mu_2}(f).$$
\end{lemma}
\begin{proof} 
By \eqref{formula-D} and \cite[Proposition~3]{S1997}, 
\beqn
D_{\mu_1, \mu_2}(f_R) =  \lim_{r \rar 1^{-}} \int_{\mathbb T} D_{\mu_1}(f_R(\cdot, re^{i\theta})) d\theta 
+ \lim_{r \rar 1^{-}} \int_{\mathbb T} D_{\mu_2}(f_R(re^{i\theta}, \cdot))d\theta \\
\Le \lim_{r \rar 1^{-}} \int_{\mathbb T} D_{\mu_1}(f(\cdot, rR_2e^{i\theta})) d\theta 
+ \lim_{r \rar 1^{-}} \int_{\mathbb T} D_{\mu_2}(f(rR_1e^{i\theta}, \cdot))d\theta. 
\eeqn 
This, combined with Lemma~\ref{increasing-lem}, yields
\beqn
D_{\mu_1, \mu_2}(f_R) 
\Le \lim_{r \rar 1^{-}} \int_{\mathbb T} D_{\mu_1}(f(\cdot, re^{i\theta})) d\theta 
+ \lim_{r \rar 1^{-}} \int_{\mathbb T} D_{\mu_2}(f(re^{i\theta}, \cdot))d\theta.
\eeqn 
An application of \eqref{formula-D} now completes the proof. 
\end{proof}

Here is a key step in deducing the density of polynomials in $\mathcal D(\mu_1, \mu_2).$ 
\begin{lemma} \label{poly-density-lem}
For any $f \in \mathcal D(\mu_1, \mu_2),$ 
$$\lim_{R_1, R_2 \rar 1^{-}}D_{\mu_1, \mu_2}(f-f_{R}) =0.$$
\end{lemma}
\begin{proof} The proof is an adaptation of that of \cite[Theorem~7.3.1]{EKMR2014} to the present situation. 
For $R=(R_1, R_2) \in (0, 1)^2,$ by the Parallelogram law (which holds for any seminorm) and Lemma \ref{contractivity-D}, 
\beq \label{para-law}
D_{\mu_1, \mu_2}(f-f_{R}) + D_{\mu_1, \mu_2}(f+f_{R}) &=& 2(D_{\mu_1, \mu_2}(f) + D_{\mu_1, \mu_2}(f_{R})) \notag \\
&\Le & 4D_{\mu_1, \mu_2}(f).
\eeq
We claim that
\beq \label{claim}
\liminf_{R_1, R_2 \rar 1^{-}} D_{\mu_1, \mu_2}(f+f_{R}) \Ge 4D_{\mu_1, \mu_2}(f).
\eeq 
To see this, fix $r \in (0, 1).$ By Fatou's lemma, 
\beq \label{appli-Fatou}
&&  \liminf_{R_1, R_2 \rar 1^{-}} \Big(\int_{\mathbb T} D_{\mu_1}((f+f_R)(\cdot, re^{i\theta})) d\theta 
+   \int_{\mathbb T} D_{\mu_2}((f+f_R)(re^{i\theta}, \cdot))d\theta\Big) \notag \\
& \Ge & 4\Big(\int_{\mathbb T} D_{\mu_1}(f(\cdot, re^{i\theta})) d\theta 
+    \int_{\mathbb T} D_{\mu_2}(f(re^{i\theta}, \cdot))d\theta\Big).
\eeq
On the other hand, by Lemma~\ref{increasing-lem},
\beqn
&& D_{\mu_1, \mu_2}(f+f_{R}) \\
&\Ge & \int_{\mathbb T} D_{\mu_1}((f+f_R)(\cdot, re^{i\theta})) d\theta 
+ \int_{\mathbb T} D_{\mu_2}((f+f_R)(re^{i\theta}, \cdot))d\theta. 
\eeqn
After taking $\liminf$ on both sides (one by one) and applying \eqref{appli-Fatou}, we get
\beqn
&& \liminf_{R_1, R_2 \rar 1^{-}} D_{\mu_1, \mu_2}(f+f_{R}) \\
& \Ge & 
 4\Big(\int_{\mathbb T} D_{\mu_1}(f(\cdot, re^{i\theta})) d\theta 
+    \int_{\mathbb T} D_{\mu_2}(f(re^{i\theta}, \cdot))d\theta\Big).
\eeqn
Letting $r \rar 1^{-}$ on the right-hand side now yields \eqref{claim} (see \eqref{formula-D}). Finally, note that by \eqref{para-law},
\beqn
\limsup_{R_1, R_2 \rar 1^{-}} D_{\mu_1, \mu_2}(f-f_{R}) \Le 4D_{\mu_1, \mu_2}(f) - \liminf_{R_1, R_2 \rar 1^{-}} D_{\mu_1, \mu_2}(f+f_{R}),
\eeqn
and hence by \eqref{claim}, we get  $$\limsup_{R_1, R_2 \rar 1^{-}} D_{\mu_1, \mu_2}(f-f_{R})=0,$$ which completes the proof.   
\end{proof}

We now complete the proof of Theorem~\ref{thm1}.

\begin{proof}[Proof of Theorem~\ref{thm1}] Parts (i) and (iii) are Lemmas~\ref{multiplier} and \ref{Richter-formula}, respectively. To see (ii), 
let $f \in \mathcal D(\mu_1, \mu_2)$ and $\epsilon >0.$ It suffices to check that
there exists a polynomial $p$ in $z_1$ and $z_2$ such that 
$\|f-p\|_{\mathcal D(\mu_1, \mu_2)} < \epsilon.$
It is easy to see using Lemma \ref{poly-density-lem} that there exist an $R=(R_1, R_2) \in (0, 1)^2$ such that 
\beq \label{estimate-1}
\|f-f_{R}\|_{\mathcal D(\mu_1, \mu_2)} < \epsilon/2.
\eeq
Since $f_R$ is holomorphic in an open neighborhood of ${\overline{\mathbb D}}^2,$ there exists a polynomial $p$ such that $$\|\partial^{\alpha}f_R-\partial^{\alpha}p\|_{\infty, {\overline{\mathbb D}^2}} < \frac{\sqrt{\epsilon}}{4\sqrt{M}}, \quad \alpha \in \{(0, 0), (1, 0), (0, 1)\},$$ where $M=\max\big\{\int_{\mathbb D}P_{\mu_j}(w)dA(w) : j=1, 2\big\}+1.$ This together with the fact that the norm on $H^2(\mathbb D^2)$ is dominated by the $\|\cdot\|_{\infty, \overline{\mathbb D}^2}$ shows that $\|f_{R}-p\|_{\mathcal D(\mu_1, \mu_2)} < \epsilon/2.$ Combining this with \eqref{estimate-1} yields $\|f-p\|_{\mathcal D(\mu_1, \mu_2)} < \epsilon,$ which completes the proof.   
\end{proof}

The following provides a ground to discuss operator theory on $\mathcal D(\mu_1, \mu_2).$
\begin{corollary} \label{cyclic-coro} For $j=1, 2,$ let $\mathscr M_{z_j}$ denote the operator of multiplication by the coordinate function $z_j.$ Then 
the commuting pair $\mathscr M_z=(\mathscr M_{z_1}, \mathscr M_{z_2})$ on $\mathcal D(\mu_1, \mu_2)$ is a cyclic toral $2$-isometry with cyclic vector $1.$ 
\end{corollary}
\begin{proof}
Note that by Theorem~\ref{thm1}(i) and the closed graph theorem, $\mathscr M_z$ defines a pair of bounded linear operators $\mathscr M_{z_1}$ and $\mathscr M_{z_2}$ on $\mathcal D(\mu_1, \mu_2).$ 
By Theorem~\ref{thm1}(ii), $\mathscr M_z$ is cyclic with cyclic vector $1.$ 
Finally, the fact that $\mathscr M_z$ is a toral $2$-isometry may be derived from (ii) and (iii) of Theorem~\ref{thm1}.
\end{proof}

Let $\kappa : \mathbb D \times \mathbb D \rar \mathbb C$ denote the reproducing kernel of 
$\mathcal D(\mu_1, \mu_2)$ (see Lemma~\ref{r-kernel}). 
\begin{corollary} \label{T-spectrum}  For any $w \in \mathbb D^2,$ $\ker(\mathscr M_z- w)=\{0\}$ and $\ker(\mathscr M^*_z - w)$ is the one-dimensional space spanned by $\kappa(\cdot, \overline{w}).$ 
\end{corollary}
\begin{proof} Since $\mathcal D(\mu_1, \mu_2)$ is contained in the space of complex-valued holomorphic functions on $\mathbb D^2,$ the pair $\mathscr M_z$ has no eigenvalue. 
By Theorem~\ref{thm1}, $\mathscr M_z$ is cyclic, and hence, for any $w \in \mathbb C^2,$ the dimension of $\ker(\mathscr M^*_z - w)$ is at most $1$ (see \eqref{cyclic-kernel}). If $w \in \mathbb D^2,$ then by the reproducing property of $\mathcal D(\mu_1, \mu_2)$ (see Lemma~\ref{r-kernel}), 
$\kappa(\cdot, \overline{w}) \in \ker(\mathscr M^*_z - w).$ Since $1 \in \mathcal D(\mu_1, \mu_2),$ once again by the reproducing property of $\mathcal D(\mu_1, \mu_2),$
$\kappa(\cdot, \overline{w})\neq 0.$ 
\end{proof}

Before we state the next application of Theorem~\ref{thm1}, we need a formula for the inner-product of monomials in $\mathcal D(\mu_1, \mu_2).$ 
\begin{proposition}
For $\mu \in M_+(\mathbb T)$ and $j \Ge 0,$ let $\hat{\mu}(j) = \int_{\mathbb T} \zeta^{-j} d\mu(\zeta).$ Then 
\beq \label{inner-p-formula}
\inp{z^{m}_1z^n_2}{z^{p}_1z^q_2}_{\mathcal D(\mu_1, \mu_2)} = \begin{cases}
0 & \mbox{if}~ m\neq p, ~n \neq q, \\
\min\{n, q\}\hat{\mu}_2(q-n) & \mbox{if}~ m=p, ~n \neq q, \\
\min\{m, p\}\hat{\mu}_1(p-m) & \mbox{if}~ m \neq p, ~n = q, \\
1+m\hat{\mu}_1(0) + n\hat{\mu}_2(0) & \mbox{if}~ m= p, ~n = q.
\end{cases}
\eeq 
In particular, the monomials are orthogonal in $\mathcal D(\mu_1, \mu_2)$ if and only if $\mu_1$ and $\mu_2$ are nonnegative multiples of the Lebesgue measure on $\mathbb T.$
\end{proposition}
\begin{proof}
Fix non-negative integers $m, n, p, q.$ 
By the polarization identity, 
\beqn
&& \inp{z^{m}_1z^n_2}{z^{p}_1z^q_2}_{\mathcal D(\mu_1, \mu_2)} = \inp{z^{m}_1z^n_2}{z^{p}_1z^q_2}_{H^2(\mathbb D^2)} \\
&+& \lim_{r \rar 1^{-}} \int_{\mathbb T}  r^{n+q} e^{i (n-q) \theta} \int_{\mathbb D} \partial_1(z^m_1)   \, \overline{\partial_1(z^p_1)} P_{\mu_1}(z_1) \,dA(z_1)d\theta \\
&+& \lim_{r \rar 1^{-}} \int_{\mathbb T}  r^{m+p}e^{i (m-p) \theta} \int_{\mathbb D} \partial_2(z^n_2)   \, \overline{\partial_2(z^q_2) } P_{\mu_2}(z_2) \,dA(z_2)d\theta.
\eeqn
Since $\inp{z^{m}_1z^n_2}{z^{p}_1z^q_2}_{H^2(\mathbb D^2)}=\delta(m, p)\delta(n, q)$ with $\delta(\cdot, \cdot)$ denoting the Kronecker delta of two variables, \eqref{inner-p-formula} may be deduced from the following formula for the inner-product of the Dirichlet-type space $\mathcal D(\mu)$ (see \cite[Equation~(3.2)]{O2004}): 
\beq \label{formula-o} \langle z^r, z^s\rangle_{\mathcal D(\mu)}=\delta(r,s) + \min\{r, s\}
\hat{\mu}(s-r), \quad r, s \in \mathbb Z_+
\eeq
(this formula may also be derived directly using \cite[Theorem~11.9]{Ru1987}). 
The ``In particular'' part follows from the Weierstrass approximation theorem and Riesz representation theorem.   
\end{proof}
\begin{remark} Assume that $\mu_1, \mu_2$ are non-zero. 
It is easy to see using \eqref{inner-p-formula} and \eqref{formula-o} that 
$\|f\|_{\mathcal D(\mu_1, \mu_2)} = \|f\|_{\mathcal D(\mu_1) \otimes \mathcal D(\mu_2)}$ holds for all monomials $f$ if
and only if at least one of $\mu_1$ and $\mu_2$ is the zero measure.
In particular, $\mathcal D(\mu_1, \mu_2) \neq \mathcal D(\mu_1) \otimes \mathcal D(\mu_2),$ in general. 
\end{remark}

The following is a consequence of \eqref{inner-p-formula} (see Definition~\ref{def-ws}). 
\begin{corollary} \label{wandering-s} 
For $\mu_1, \mu_2 \in M_+(\mathbb T),$
 the subspace of $\mathcal D(\mu_1, \mu_2)$ spanned by the constant function $1$ is a wandering subspace for $\mathscr M_z$ on $\mathcal D(\mu_1, \mu_2).$
\end{corollary}

A bounded linear operator $T$ on a Hilbert space is {\it essentially normal} if $T^*T-TT^*$ is a compact operator. 
An essentially normal operator is said to be {\it essentially unitary} if $T^*T-I$ is compact. 
Unlike the case of one variable Dirichlet-type spaces (see \cite[Proposition~2.21]{C2007}), $\mathcal D(\mu_1, \mu_2)$ does not support essentially normal multiplication $2$-tuple $\mathscr M_z.$  
\begin{corollary} \label{coro-e-n} 
The multiplication operators $\mathscr M_{z_1}$ and $\mathscr M_{z_2}$ on $\mathcal D(\mu_1, \mu_2)$ are never essentially normal.  
\end{corollary}
\begin{proof} 
By Theorem~\ref{thm1}, the multiplication $2$-tuple $\mathscr M_z$ is a toral $2$-isometry. In particular, $\mathscr M_{z_1}$ and $\mathscr M_{z_2}$ are $2$-isometries. If these are essentially normal, then 
the image of $\mathscr M_{z_1}$ and $\mathscr M_{z_2}$ in the Calkin algebra is a normal $2$-isometry, and hence $\mathscr M_{z_1}$ and $\mathscr M_{z_2}$ are essentially unitary (since a normal $2$-isometry, being invertible, is a unitary). It follows that  $\mathscr M_{z_1}$ and $\mathscr M_{z_2}$ are Fredholm. In view of Atkinson's theorem (see \cite[Theorem~XI.2.3]{Co1990}), it suffices to check that the kernels of $\mathscr M^*_{z_1}$ and $\mathscr M^*_{z_2}$ are of infinite dimension.  To see this, fix a nonnegative integer $j.$ Note that by \eqref{inner-p-formula},
\beqn
\inp{\mathscr M^*_{z_1}z^j_2}{z^p_1z^q_2} = \inp{z^j_2}{z^{p+1}_1z^q_2}=0, \quad p, q \in \mathbb Z_+,
\eeqn
and hence by the linearity of the inner-product and the density of the polynomials in $\mathcal D(\mu_1, \mu_2)$ (see Theorem~\ref{thm1}(ii)), we obtain $\mathscr M^*_{z_1}z^j_2=0.$ Similarly, one can check that $z^j_1 \in \ker \mathscr M^*_{z_2},$ completing the proof. 
\end{proof}


\section{Proof of Theorem~\ref{D-mu-estimate} and a consequence} 


We begin the proof of Theorem~\ref{D-mu-estimate} with the following special case.

\begin{lemma} \label{d-property}
The Hardy space $H^2(\mathbb D^d)$ has the division property. 
\end{lemma}
\begin{proof} For $j=1, \ldots, d$ and $\lambda=(\lambda_1, \ldots, \lambda_d) \in \mathbb D^d,$  
let $f \in H^2(\mathbb D^d)$ be such that $\{z \in \mathbb D^d : z_j =\lambda_j\} \subseteq Z(f).$ 
Let $w=(w_1, \ldots, w_d) \in \mathbb D^d.$ If $w_j \neq \lambda_j,$ then clearly $g_j(z) =\frac{f(z)}{z_j-\lambda_j}$ defines a holomorphic function in a neighborhood of $w.$
If $w_j =\lambda_j,$ then since $\{z \in \mathbb D^d : z_j =\lambda_j\} \subseteq Z(f),$ $f$ as a function of $z_j$ has a removable singularity at $w_j,$ and hence by Hartogs' separate analyticity theorem (see \cite[pp~1-2]{Ru1969}), $g_j$ is holomorphic in a neighborhood of $w.$ This shows that $g_j$ is holomorphic on $\mathbb D^d.$ 
To see that $g_j \in H^2(\mathbb D^d),$  
note that for $r \in (|\lambda_j|, 1)$ and $\theta_1, \ldots, \theta_d \in [0, 2\pi],$  
\beqn
\frac{|f(re^{i \theta_1}, \ldots, re^{i \theta_1})|}{|re^{i \theta_j}-\lambda_j|}
\Le  \frac{|f(re^{i \theta_1}, \ldots, re^{i \theta_1})|}{r-|\lambda_j|}. 
\eeqn
Since $f \in H^2(\mathbb D^d),$ it may now be deduced from \eqref{Hardy-norm-new} that $g_j \in H^2(\mathbb D^d).$  
\end{proof}

We also need the following fact essentially noticed in \cite{R1991}. 
\begin{lemma} 
\label{D-for-D(mu)}
For any $\mu \in M_+(\mathbb T),$ $\mathcal D(\mu)$ has the division property.
\end{lemma}
\begin{proof} For $\lambda \in \mathbb D,$ let $g \in \mathcal D(\mu)$ be such that $g(\lambda)=0.$ Note that $g$ is orthogonal to $\kappa(\cdot, \lambda).$ Since $\ker(\mathscr M^*_z-\overline{\lambda})$ is spanned by 
$\kappa(\cdot, \lambda)$ and the range of $\mathscr M_z-\lambda$ is closed (see \cite[Corollary~3.8]{R1991}), there exists $f \in \mathcal D(\mu)$ such that $g=(z-\lambda)f,$ which completes the proof. 
\end{proof}
\begin{proof}[Proof of Theorem~\ref{D-mu-estimate}]  For $\lambda \in \mathbb D,$ 
assume that $(z_j-\lambda)h \in \mathcal D(\mu_1, \mu_2)$ for some $j=1, 2.$  Thus  
\beq \label{D-mu-finite-new} 
&& (z_j -\lambda)h \in H^2(\mathbb D^2) \\
&& D_{\mu_1, \mu_2}((z_j -\lambda)h) < \infty. \label{D-mu-finite} 
\eeq 
Since the arguments for the cases $j=1, 2$ are similar, we only treat the case when $j=1.$ 
It follows from Lemma~\ref{d-property} and \eqref{D-mu-finite-new} that $h \in H^2(\mathbb D^2).$
Applying \eqref{rmk-slice-theta} to \eqref{D-mu-finite} gives 
\beq \label{D-mu-member}
D_{\mu_1}( (z_1 -\lambda)h(\cdot, re^{i \theta})) < \infty, \quad r \in (0, 1),~ \theta \in \Omega_r,
\eeq
where $\Omega_r$ is a Lebesgue measurable subset of $[0, 2\pi]$ such that $[0, 2\pi] \backslash \Omega_r$ is of measure $0.$
For $r \in (0, 1)$ and $\theta \in \Omega_r,$ 
consider $f_{r, \theta} : \mathbb D \rar \mathbb C$ defined by 
\beqn
f_{r, \theta}(w)=(w-\lambda)h(w, re^{i \theta}), \quad w \in \mathbb D.
\eeqn
By \eqref{D-mu-finite-new} and Lemma~\ref{Hardy-slice}, $f_{r, \theta}$ belongs to $H^2(\mathbb D).$ Hence, 
by \eqref{D-mu-member}, $f_{r, \theta}$ belongs to $\mathcal D(\mu_1).$ Hence, by Lemma~\ref{D-for-D(mu)}, 
$h(\cdot, re^{i \theta}) \in \mathcal D(\mu_1).$ Since $$\|h(\cdot, re^{i\theta})\|_{\mathcal D(\mu_1)} \Le \|wh(\cdot, re^{i\theta})\|_{\mathcal D(\mu_1)}$$ (see \cite[Theorem~3.6]{R1991}), by the reverse triangle inequality,  
\beqn
\|f_{r, \theta}\|^2_{\mathcal D(\mu_1)} \Ge (1-|\lambda|)^2\|h(\cdot, re^{i \theta})\|^2_{\mathcal D(\mu_1)} \Ge (1-|\lambda|)^2 D_{\mu_1}(h(\cdot, re^{i \theta})).
\eeqn
Integrating both sides over $[0, 2\pi]$ yields
\beqn
(1-|\lambda|)^2\int_{0}^{2\pi} D_{\mu_1}(h(\cdot, re^{i \theta}))d\theta & \Le & \int_{0}^{2\pi}\|f_{r, \theta}\|^2_{H^2(\mathbb D)}d\theta +  \int_{0}^{2\pi}D_{\mu_1}(f_{r, \theta})d\theta. 
\\ 
& \Le & \|(z_1-\lambda)h\|^2_{H^2(\mathbb D^2)} + D_{\mu_1, \mu_2}((z_1 -\lambda)h),
\eeqn 
where we used \eqref{Hardy-fact}.
Taking supremum over $r \in (0, 1)$ gives now
\beqn
\sup_{0 < r < 1}\int_{0}^{2\pi} D_{\mu_1}(h(\cdot, re^{i \theta}))d\theta < \infty. 
\eeqn
Also, since $h \in H^2(\mathbb D^2),$ 
it now suffices to check that 
 \beq \label{suffices-h-1}
\sup_{0 < r < 1} \int_{0}^{2\pi} D_{\mu_2}(h(re^{i \theta}, \cdot))d\theta < \infty. 
\eeq
Note that by \eqref{D-mu-finite}, 
\beqn
\sup_{0 < r < 1} \int_{0}^{2\pi} D_{\mu_2}\Big(((z_1-\lambda)h)(re^{i \theta}, \cdot)\Big)d\theta < \infty. 
\eeqn
However, for any $s \in (|\lambda|, 1),$ 
\beqn
&& \sup_{0 < r < 1}  \int_{0}^{2\pi} D_{\mu_2}\Big(((z_1-\lambda)h)(re^{i \theta}, \cdot)\Big)d\theta \\
& \Ge & \int_{0}^{2\pi} D_{\mu_2}\Big(((z_1-\lambda)h)(se^{i \theta}, \cdot)\Big)d\theta \\
& \Ge & (s-|\lambda|)^2 \int_{0}^{2\pi} \int_{\mathbb D}  |\partial_2 h(se^{i \theta}, w)|^2 P_{\mu_2}(w)dA(w)d\theta.
\eeqn
Applying Lemma~\ref{increasing-lem} and letting $s \uparrow 1^{-}$ now yields \eqref{suffices-h-1}. 
\end{proof}

Before we present an application of Theorem~\ref{D-mu-estimate}, let us recall some facts from the multivariate spectral theory (see \cite{Cu1981, Cu1988, T1970}). Let $T=(T_1, T_2)$ be a commuting pair  on $\mathcal H$
and set $D_T(x)=(T_1x, T_2x),$ $x \in \mathcal H.$ Note that
\beq 
\label{column-o}
\mbox{if $D^*_TD_T$ is Fredholm, then $D_T$ has closed range.} 
\eeq 
Indeed, 
if $D^*_TD_T$ is Fredholm, then $D_T$ is left-Fredholm, and hence we obtain \eqref{column-o}. 
To define the Taylor spectrum, we consider the following complex:
\beq \label{K} K(T, \mathcal H) : \lbrace 0 \rbrace \xrightarrow{0} \mathcal H \xrightarrow{B_2} \mathcal H \oplus \mathcal H \xrightarrow{B_1} \mathcal H \xrightarrow{0} \lbrace 0 \rbrace, \eeq
where the {\it boundary maps} $B_1$ and $B_2$ are given by
\beqn B_2(h) :=  (T_2 h, -T_1 h), \quad B_1(h_1,h_2) :=  T_1 h_{1}+T_2 h_{2}.\eeqn
Note that $K(T, \mathcal H)$ is a complex, that is, $ B_1 \circ B_2 = 0.$
Let $H^k({T})$ denote the $k$-th cohomology group in $K({T},\mathcal{H}),$ $k=0, 1, 2.$ Note that $H^0({T}) = \ker(T)$ and $H^2({T}) = \ker T^*$.
Following \cite{T1970} (resp. \cite{Cu1981}), we say that $T$ is {\it Taylor-invertible} (resp. {\it Fredholm}) if $H^k({T})=\{0\}$ (resp. $\dim H^k({T})<\infty$) for $k=0, 1, 2.$  The {\it Taylor spectrum} $\sigma(T)$ and
the {\it essential spectrum} $\sigma_e(T)$ are given by 
\beqn
\sigma(T)&=&\{\lambda \in \mathbb C^2 : T-\lambda~\mbox{is not Taylor-invertible}\}, \\ \sigma_e(T)&=&\{\lambda \in \mathbb C^2 : T-\lambda~\mbox{is not Fredholm}\}.
\eeqn
The {\it Fredholm index} $\mbox{ind}({T})$ of a commuting $2$-tuple ${T}$ on $\mathcal H$ is the {\it Euler characteristic} of the Koszul complex $K({T}, \mathcal{H}),$ that is,  
\beq
\label{index} \mbox{ind}({T}):= \dim H^0({T}) - \dim H^1({T}) + \dim H^2({T}).
\eeq

As an application of the division property, we now show that we always have exactness at the middle stage of the Koszul complex of the multiplication $2$-tuple $\mathscr M_z$ on $\mathcal D(\mu_1, \mu_2).$ First a general fact. 
\begin{lemma} \label{weak-strong-Gleason}
Let $\mathscr H$ be a reproducing kernel Hilbert space of complex-valued holomorphic functions on the unit bidisc $\mathbb D^2.$ Assume that $\mathscr M_z=(\mathscr M_{z_1}, \mathscr M_{z_2})$ is a commuting pair on $\mathscr H.$ If $\mathscr H$ has the division property, then for every $\lambda = (\lambda_1, \lambda_2) \in \mathbb D^2,$ the Koszul complex of $\mathscr M_z-\lambda = (\mathscr M_{z_1}-\lambda_1, \mathscr M_{z_2}-\lambda_2)$ is exact at the middle stage $($see \eqref{K}$).$ 
\end{lemma}
\begin{proof} 
Note that $\mathscr H$ has the division property if and only if 
for $j=1, 2,$ we have the following property:
\beq \label{w-Gleason-property}
&&\mbox{for any holomorphic function
$h : \mathbb D^2 \rar \mathbb C$ and 
$\lambda \in \mathbb D^2,$} \notag \\
&&\mbox{if $(z_j-\lambda_j)h \in \mathscr H,$  then $h \in \mathscr H.$} 
\eeq
We first assume that \eqref{w-Gleason-property} holds for $j=2$. 
To see that the Koszul complex of $\mathscr M_z-\lambda$ is exact at the middle stage, let $g, h \in \mathscr H$ be such that 
\beq \label{exactness-a-new}
(z_2-\lambda_2)g(z_1, z_2)=(z_1 - \lambda_1)h(z_1, z_2), \quad (z_1, z_2) \in \mathbb D^2.
\eeq
Letting $z_2=\lambda_2,$ we obtain $(w-\lambda_1)h(w, \lambda_2)=0$ for every $w \in \mathbb D.$ It follows that $h(\cdot, \lambda_2)=0$ on $\mathbb D.$ Since $h : \mathbb D^2 \rar \mathbb C$ is holomorphic, there exists a holomorphic function $k : \mathbb D^2 \rar \mathbb C$ such that  
\beq \label{exactness-b-new}
h(z_1, z_2)=(z_2 -\lambda_2)k(z_1, z_2), \quad (z_1, z_2) \in \mathbb D^2
\eeq
(in case of $\lambda_2=0,$ this can be seen using the power series for $h$; the general case can be dealt now by replacing $h(z_1, z_2)$ by $h(z_1, \varphi(z_2)),$ where $\varphi$ is the automorphism of $\mathbb D$ which takes $\lambda_2$ to $0$).
Since $h \in \mathscr H,$ by \eqref{w-Gleason-property}, $k \in \mathscr H.$ We now combine \eqref{exactness-a-new} with \eqref{exactness-b-new}  to obtain 
\beqn
(z_2-\lambda_2)g(z_1, z_2)&=&(z_1 - \lambda_1)h(z_1, z_2) \\
&=& (z_1 - \lambda_1)(z_2 -\lambda_2)k(z_1, z_2), \quad z \in \mathbb D^2.
\eeqn
This gives $g(z_1, z_2)=(z_1-\lambda_1) k(z_1, z_2),$ $z \in \mathbb D^2.$
This together with \eqref{exactness-b-new} shows that $\mathscr M_z-\lambda$ is exact at the middle stage. To complete the proof, we obtain the same conclusion in case \eqref{w-Gleason-property} holds for $j=1$. Indeed, one may proceed as above with the only change that the roles of $\lambda_1$ and $\lambda_2$ are interchanged (e.g. \eqref{exactness-a-new} is evaluated at $z_1=\lambda_1$).  
\end{proof}
\begin{remark} 
Let $\Omega$ be a bounded domain. 
One may imitate the first part of the proof of Lemma~\ref{d-property} to show that there exists a holomorphic function $k : \Omega \rar \mathbb C$ satisfying \eqref{exactness-b-new}. This gives an analog of Lemma~\ref{weak-strong-Gleason} for arbitrary bounded domains.
\end{remark}

The following is a consequence of Theorem~\ref{D-mu-estimate} and Lemma~\ref{weak-strong-Gleason}.
\begin{corollary} \label{exact-middle}
For every $\lambda = (\lambda_1, \lambda_2) \in \mathbb D^2,$ the Koszul complex of the $2$-tuple $\mathscr M_z-\lambda = (\mathscr M_{z_1}-\lambda_1, \mathscr M_{z_2}-\lambda_2)$ on $\mathcal D(\mu_1, \mu_2)$ is exact at the middle stage $($see \eqref{K}$).$
\end{corollary}

\section{Proof of Theorem~\ref{Gleason-new} and its consequences}

We begin with a
 lemma, which is a variant of \cite[Lemma~4.14]{GM2020}. We include its proof for the sake of completeness.  
\begin{lemma} \label{abstract-Gleason-p}
For a domain $\Omega$ of $\mathbb C^2,$ 
let $\mathscr H$ be the reproducing kernel Hilbert space of complex-valued holomorphic functions associated with the  kernel $\kappa : \Omega \times \Omega \rar \mathbb C.$ Assume that the constant function $1$ belongs to $\mathscr H,$ the multiplication operators $\mathscr M_{z_1}, \mathscr M_{z_2}$ are bounded on $\mathscr H$ and the commuting $2$-tuple $\mathscr M_z$ is cyclic. For $w \in \Omega,$  Gleason's problem can be solved for $\mathscr H$ over $\{w\}$ if and only if
\beq \label{CRT}
\mbox{$D^*_{\mathscr M^*_z-\overline{w}}$ has closed range.}
\eeq
In particular,  Gleason's problem can be solved for $\mathscr H$ over 
$\Omega \backslash \sigma_e(\mathscr M_z).$
\end{lemma}
\begin{proof} 
Let $w \in \Omega$ and let $f \in \mathscr H.$ By the reproducing kernel property of $\mathscr H,$ \beq \label{f-minus-f(w)}
f-f(w) \in \{c\kappa(\cdot, w) :  c \in \mathbb C\}^{\perp}.
\eeq 
However, since $\mathscr M_z$ is cyclic, $\dim \ker(\mathscr M^*_z-\overline{w}) \Le 1$ for every $w \in \mathbb C^2$ (see \eqref{cyclic-kernel}). As $1 \in \mathscr H,$ we have $\kappa(\cdot, w) \neq 0,$ and hence 
\beqn 
\{c\kappa(\cdot, w) :  c \in \mathbb C\} = \ker(\mathscr M^*_z-\overline{w})=\ker D_{\mathscr M^*_z-\overline{w}}.
\eeqn
It now follows from \eqref{f-minus-f(w)} that  
\beq 
\label{eq2}
 f-f(w) \in (\ker D_{\mathscr M^*_z-\overline{w}})^{\perp}=\overline{\mbox{ran}(D^*_{\mathscr M^*_z-\overline{w}})}.
\eeq
Also, it is easy to see that
\beq
\label{eq3}
\mbox{ran}(D^*_{\mathscr M^*_z-\overline{w}})  
= \{(z_1-w_1)g_1 + (z_2-w_2)g_2 : g_1, g_2 \in \mathscr H\},
\eeq
If \eqref{CRT} holds, then it now follows from \eqref{eq2} that 
\beqn 
 f-f(w) \in \{(z_1-w_1)g_1 + (z_2-w_2)g_2 : g_1, g_2 \in \mathscr H\},
\eeqn
and hence Gleason's problem can be solved for $\mathscr H$ over $\{w\}.$ Conversely, if Gleason's problem can be solved for $\mathscr H$ over $\{w\},$ then by \eqref{eq2}, any function in $\overline{\mbox{ran}(D^*_{\mathscr M^*_z-\overline{w}})}$ is of the form $f-f(w)$ for some $f \in \mathscr H,$ and hence by \eqref{eq3}, it belongs to $\mbox{ran}(D^*_{\mathscr M^*_z-\overline{w}}).$
This completes the proof of the equivalence. 

To see the remaining part, let $w = (w_1, w_2) \in \Omega \backslash  \sigma_e(\mathscr M_z).$ 
Since $D^*_SD_S = S^*_1S_1 + S^*_2S_2$ for any commuting pair $S=(S_1, S_2),$ 
by \cite[Corollary~3.6]{Cu1981}, the operator $D^*_{\mathscr M^*_z-\overline{w}}D_{\mathscr M^*_z-\overline{w}}$ is Fredholm, and hence by \eqref{column-o}, 
$D_{\mathscr M^*_z-\overline{w}}$ has closed range. 
Hence, by the closed-range theorem (see \cite[Theorem~VI.1.10]{Co1990}), we obtain \eqref{CRT} completing the proof.
\end{proof}
\begin{remark} \label{rmk-abstract-G}
Let  $\mathscr M_z$ be the multiplication $2$-tuple on the Hardy space $H^2(\mathbb D^2)$ of the unit bidisc $\mathbb D^2.$ Since $\sigma_e(\mathscr M_z) \cap \mathbb D^2 = \emptyset$ (see \cite[Theorem~5(c)]{Cu1981}), by Lemma~\ref{abstract-Gleason-p},  Gleason's problem can be solved for $H^2(\mathbb D^2).$ 
\end{remark}

The following lemma provides a situation in which the division property ensures a solution to Gleason's problem. 
\begin{lemma}
\label{division-slice}
Let $\mathscr H$ be a reproducing kernel Hilbert space of complex-valued holomorphic functions on the unit bidisc $\mathbb D^2$ and let $w=(w_1, w_2) \in \mathbb D^2.$ Assume that $\mathscr H$ has the division property and $\mathscr M_z=(\mathscr M_{z_1}, \mathscr M_{z_2})$ is a commuting pair on $\mathscr H.$ If, for every $f \in \mathscr H,$ either $f(\cdot, w_2)$ or $f(w_1, \cdot)$ belongs to $\mathscr H,$ 
then 
Gleason's problem can be solved for $\mathscr H$ over $\{w\}.$
\end{lemma}
\begin{proof} For $f \in \mathscr H,$ assume that $f(w_1, \cdot) \in \mathscr H.$  
Thus $f-f(w_1, \cdot)\in \mathscr H.$ Hence, if 
$h : \mathbb D^2 \rar \mathbb C$ is a holomorphic function such that 
\beq 
\label{observation-h}
f(z_1, z_2)-f(w_1, z_2)=(z_1-w_1)h(z_1, z_2), \quad z_1, z_2 \in \mathbb D,
\eeq
by the division property for $\mathscr H,$ we have $h \in \mathscr H.$
Also, since $f(w_1, \cdot) \in \mathscr H,$ one may argue as above to see that there exists $k \in \mathscr H$ satisfying 
\beqn 
f(w_1, z_2)-f(w_1, w_2)=(z_2-w_2)k(z_1, z_2), \quad z_1, z_2 \in \mathbb D..
\eeqn
This, combined with \eqref{observation-h}, completes the proof in this case. Similarly, one can deal with the case in which $f(\cdot, w_2) \in \mathscr H.$   
\end{proof}


We also need the following fact of independent interest:
\begin{lemma} \label{contractive-h}
For every $f \in \mathcal D(\mu_1, \mu_2),$ the slice functions $f(\cdot, 0)$ and $f(0, \cdot)$ belong to $\mathcal D(\mu_1, \mu_2).$ Moreover, the mappings $f \mapsto f(\cdot, 0)$ and $f \mapsto f(0, \cdot)$ from $\mathcal D(\mu_1, \mu_2)$ into itself are contractive homomorphisms. 
\end{lemma}
\begin{proof}
If $f \in \mathcal D(\mu_1, \mu_2),$ then 
\beqn
&& \int_{\mathbb T} \int_{\mathbb D}|\partial_1 f(z_1, 0)|^2 P_{\mu_1}(z_1) \,dA(z_1)d\theta \\
&+&  \int_{\mathbb T} \int_{\mathbb D}|\partial_2 f(0, z_2)|^2 P_{\mu_2}(z_2) \,dA(z_2)d\theta  \Le  D_{\mu_1, \mu_2}(f).
\eeqn 
Since the mappings $f \mapsto f(\cdot, 0)$ and $f \mapsto f(0, \cdot)$ from $H^2(\mathbb D^2)$ into itself are contractive homomorphisms, the desired conclusions may be deduced from the estimate above. 
\end{proof}

\begin{proof}[Proof of Theorem~\ref{Gleason-new}] Since 
$\mathcal D(\mu_1, \mu_2)$ has the division property (see Theorem~\ref{D-mu-estimate}), by Lemmas~\ref{division-slice} and \ref{contractive-h}, Gleason's problem can be solved for $\mathscr H$ over $\{(0, 0)\}.$ Hence, by Lemma~\ref{abstract-Gleason-p} (which is applicable since $\mathscr M_z$ on $\mathcal D(\mu_1, \mu_2)$ is cyclic; see Theorem~\ref{thm1}), $D^*_{\mathscr M^*_z}$ has closed range (see \eqref{CRT}). It is now easy to see using Corollaries~\ref{T-spectrum} and \ref{exact-middle} that $\mathscr M_z$ is Fredholm. Since the essential spectrum is a closed subset of $\mathbb C^2$ not containing $(0, 0),$ there exists $r > 0$ such that 
$\mathbb D^2_r \subseteq \mathbb C^2 \backslash \sigma_e(\mathscr M_z).$ Another application of Lemma~\ref{abstract-Gleason-p} now completes the proof. 
%
\end{proof}
%

The following fact is implicit in the proof of Theorem~\ref{Gleason-new}. 
\begin{corollary} \label{Cowen-Douglas} 
The commuting $2$-tuple $\mathscr M^*_z$ on $\mathcal D(\mu_1, \mu_2)$ belongs to ${\bf B}_1(\mathbb D^2_r)$ for some $r >0.$ 
\end{corollary}
\begin{proof}
This may be deduced from Theorem~\ref{Gleason-new}, Lemma~\ref{r-kernel}, Corollary~\ref{T-spectrum} and Lemma~\ref{abstract-Gleason-p} (see \eqref{CRT}). 
\end{proof}

We conclude this section with the following corollary describing the cokernels of the multiplication operators $\mathscr M_{z_j},$ $j = 1, 2,$ on $\mathcal D(\mu_1, \mu_2).$ 
\begin{corollary} 
\label{description-kernel}
For $1 \Le i \neq j \Le 2,$ 
$\ker \mathscr M^*_{z_j} = \bigvee \big\{z^k_i : k \Ge 0\big\}.$ 
\end{corollary}
\begin{proof}
As observed in the proof of Corollary~\ref{coro-e-n},  
\beq \label{kernel-inclusion}
\{p(z_i) : p \in \mathbb C[w]\} \subseteq \ker \mathscr M^*_{z_j}, \quad 1 \Le i \neq j \Le 2.
\eeq
To see the reverse inclusion, let 
$f \in \ker \mathscr M^*_{z_1}.$ 
By Theorem~\ref{thm1}, there exists a sequence $\{p_n\}_{n \Ge 1}$ of complex polynomials in $z_1, z_2$ converging to $f.$ 
By Lemma ~\ref{contractive-h}, $f(0, \cdot) \in \mathcal D(\mu_1, \mu_2),$ and $\{p_n(0, \cdot)\}_{n \Ge 1}$ converges to $f(0, \cdot).$ Hence, by \eqref{kernel-inclusion}, 
$f(0, \cdot) \in \ker \mathscr M^*_{z_1}.$ Thus $f-f(0, \cdot) \in \ker \mathscr M^*_{z_1}.$ However, there exists a holomorphic function $h : \mathbb D^2 \rar \mathbb C$ such that 
\beq 
\label{existence-h-new}
f(z_1, z_2)-f(0, z_2)=z_1h(z_1, z_2), \quad z_1, z_2 \in \mathbb D. 
\eeq
By Theorem~\ref{D-mu-estimate}, 
we have $h \in \mathcal D(\mu_1, \mu_2).$ It now follows from \eqref{existence-h-new} that $\mathscr M_{z_1}h \in \ker \mathscr M^*_{z_1}.$ Since $\mathscr M^*_{z_1}\mathscr M_{z_1}$ is invertible, we must have $h=0,$ and hence, by \eqref{existence-h-new},  $f=f(0, \cdot),$ or equivalently, $f$ belongs to the closure of $\{p(z_2) : p \in \mathbb C[w]\}.$ Similarly, one can check that $\ker \mathscr M^*_{z_2}$ is equal to the closure of $\{p(z_1) : p \in \mathbb C[w]\}.$ 
\end{proof}

\section{Proof of Theorem~\ref{model-thm} and its consequences} \label{S4}

The proof of Theorem~\ref{model-thm} relies on revealing the structure of toral $2$-isometries $T$ with $\ker T^*$ as a wandering subspace (see Definition~\ref{def-ws}). 
\begin{lemma} \label{lemma-inner-formula}
Let $T=(T_1, T_2)$ be a toral $2$-isometry. Then the following statements are true:
\begin{enumerate}[\rm(i)]
\item for any integers $k, l \Ge 0,$
\beq \label{formula-moment} 
T^{*k}_1T^{*l}_2T^{l}_2T^{k}_1 &=& T^{*k}_1T^{k}_1 + T^{*l}_2T^{l}_2 - I \\ \notag
&=& kT^*_1T_1+ lT^*_2T_2 -(k+l-1)I,
\eeq
\item for $f_0 \in \ker T^*,$ assume that 
\beq \label{assumption-k-condition-1} 
&& \inp{T^m_1f_0}{T^p_1 T^q_2f_0}=0, \quad q \Ge 1, m, p \Ge 0, \\ && \inp{T^n_2f_0}{T^p_1 T^q_2f_0}=0, \quad p \Ge 1, n, q \Ge 0. \label{assumption-k-condition-2}
\eeq
Then we have the following$:$
\beqn
\inp{T^{m}_1T^n_2f_0}{T^{p}_1T^q_2f_0} = \begin{cases}
0 & \mbox{if}~ m\neq p, ~n \neq q, \\
\inp{T^n_2f_0}{T^q_2f_0} & \mbox{if}~ m=p, ~n \neq q, \\
\inp{T^{m}_1f_0}{T^{p}_1f_0} & \mbox{if}~ m \neq p, ~n = q, \\
\|T^{m}_1f_0\|^2 + \|T^n_2f_0\|^2 - \|f_0\|^2 & \mbox{if}~ m= p, ~n = q.
\end{cases}
\eeqn 
\end{enumerate}
\end{lemma}
\begin{proof}
(i) To see \eqref{formula-moment}, we proceed by strong induction on $k+l,$ $k, l \Ge 0.$ Clearly, \eqref{formula-moment} holds for $0 \Le k+l \Le 1.$ Assume that \eqref{formula-moment} holds for integers $k,l \Ge 0$ such that $0 \Le k+l \Le n.$ Note that for $k \Ge 1$ and $l \Le n-1,$ by the induction hypothesis, 
\beqn
T^*_1(T^{*k}_1T^{*l}_2T^{l}_2T^{k}_1)T_1 &=& T^{*k+1}_1T^{k+1}_1 + T^*_1T^{*l}_2T^{l}_2T_1 - T^*_1T_1 \\
&=& T^{*k+1}_1T^{k+1}_1 + T^{*l}_2T^{l}_2 - I.
\eeqn
Similarly, for $k \Le n-1$ and $l \Ge 1,$ \eqref{formula-moment} holds. This completes the induction argument. The remaining identity in (i) now follows from the known fact that for any $2$-isometry $S,$ we have 
\beq \label{moment-one}
S^{*k}S^k = k(S^*S-I)+I, \quad k \Ge 0
\eeq 
(this known fact can be seen by induction on $k \Ge 1$).  

(ii) 
Let $m, n, p, q$ be integers such that $m \neq p$ and $n \neq q.$
Consider the case when $m < p$ and $n < q.$ Since $f_0 \in \ker T^*,$ we have
\beqn
&& \inp{T^m_1 T^n_2f_0}{T^p_1 T^q_2f_0} \\
&=& \inp{T^{*n}_2 T^{*m}_1T^m_1 T^n_2f_0}{T^{p-m}_1 T^{q-n}_2f_0} \\
&\overset{\eqref{formula-moment}}=& \inp{T^{*m}_1T^m_1 f_0}{T^{p-m}_1 T^{q-n}_2f_0} + \inp{T^{*n}_2 T^n_2f_0}{T^{p-m}_1 T^{q-n}_2f_0} \\
&=& \inp{T^m_1 f_0}{T^{p}_1 T^{q-n}_2f_0} + \inp{T^n_2f_0}{T^{q}_2T^{p-m}_1 f_0},
\eeqn
which, by \eqref{assumption-k-condition-1} and \eqref{assumption-k-condition-2}, is equal to $0.$ Since the inner-product is conjugate linear, 
$\inp{T^m_1 T^n_2f_0}{T^p_1 T^q_2f_0}=0$ when $p < m$ and $q < n.$ 
Consider the case when $m < p$ and $q < n.$ Arguing as above, we have  
\beqn
&& \inp{T^m_1 T^n_2f_0}{T^p_1 T^q_2f_0} \\
&=& \inp{T^{*q}_2 T^{*m}_1T^m_1 T^{q}_2T^{n-q}_2f_0}{T^{p-m}_1 f_0} \\
&\overset{\eqref{formula-moment}\&\eqref{assumption-k-condition-2}}=&  \inp{T^{*m}_1T^m_1 T^{n-q}_2f_0}{T^{p-m}_1 f_0} + \inp{T^{*q}_2 T^{q}_2T^{n-q}_2f_0}{T^{p-m}_1 f_0} \\
&=& \inp{T^m_1 T^{n-q}_2f_0}{T^{p}_1 f_0} + \inp{T^{n}_2f_0}{T^{q}_2 T^{p-m}_1 f_0},
\eeqn
which, by \eqref{assumption-k-condition-1} and \eqref{assumption-k-condition-2}, is equal to $0.$
Once again, by the conjugate-symmetry, 
$\inp{T^m_1 T^n_2 f_0}{T^p_1 T^q_2 f_0}=0$ when 
$p < m$ and $n < q.$ 

If $m=p$ and $n \neq q,$ then one may argue as above using \eqref{moment-one} (by making cases $n < q$ and $q < n$) to show that $\inp{T^m_1 T^n_2f_0}{T^p_1 T^q_2f_0} = \inp{T^n_2f_0}{T^q_2f_0}.$ Similarly, one may derived the formula in case of $m \neq p$ and $n=q.$ Finally, if $m=p$ and $n=q,$ then the formula follows from \eqref{formula-moment}. 
\end{proof} 

\begin{proof}[Proof of Theorem~\ref{model-thm}] (i)$\Rightarrow$(iii) Fix $j \in \{1, 2\}.$ 
Since $T$ is analytic, so is $T_j.$  
Thus $T_j$ is an analytic $2$-isometry. Consider the $T_j$-invariant subspace $\mathcal H_j:=\bigvee \{T^k_j f_0 : k \Ge 0\}$ and note that
${T_j}|_{\mathcal H_j}$ is a cyclic analytic $2$-isometry. Hence, 
by \cite[Theorem~5.1]{R1991}, there exist a finite positive Borel measure $\mu_j$ on $\mathbb T$ and a unitary map $V_j : \mathcal H_j \rar \mathcal D(\mu_j)$ such that
\beq
\label{intertwining-r-new}
V_jf_0=1, \quad V_jT_j=\mathscr M^{(j)}_{w}V_j,
\eeq
where $\mathscr M^{(j)}_{w}$ denotes the operator of multiplication by the coordinate function $w$ on $\mathcal D(\mu_j).$
We contend that the map is given by
\beqn
U (T^k_1T^l_2f_0) = z^k_1 z^l_2, \quad k, l \Ge 0
\eeqn
extends to an unitary from $\mathcal H$ onto $\mathcal D(\mu_1, \mu_2).$
Since $\mathcal H=\vee \{T^{k}_1T^l_2f_0 : k, l \Ge 0\}$ and $\mathcal D(\mu_1, \mu_2)=\vee \{z^{k}_1z^l_2 : k, l \Ge 0\},$  it suffices to check that 
\beq \label{equality of norms}
\inp{T^{m}_1T^n_2f_0}{T^{p}_1T^q_2f_0} = \inp{z^m_1z^n_2}{z^p_1z^q_2}_{\mathcal D(\mu_1, \mu_2)}, \quad m, n, p, q \Ge 0. 
\eeq
Note that for any integers $m, n \Ge 0,$ by \eqref{intertwining-r-new}, 
\beqn
\inp{T^m_jf_0}{T^n_jf_0}&=& \inp{V_j T^m_jf_0}{V_jT^n_jf_0}_{\mathcal D(\mu_j)} \\
&=& \inp{(\mathscr M^{(j)}_{w})^{m}V_jf_0}{(\mathscr M^{(j)}_{w})^{n}V_jf_0}_{\mathcal D(\mu_j)} \\
&=& \inp{w^m}{w^n}_{\mathcal D(\mu_j)} \\
&=& \inp{z^m_j}{z^n_j}_{\mathcal D(\mu_1, \mu_2)}.
\eeqn
Combining this with Lemma~\ref{lemma-inner-formula}(ii) yields \eqref{equality of norms}, which completes the proof. 

(iii)$\Rightarrow$(ii) This follows from Corollaries~\ref{cyclic-coro}, \ref{wandering-s} and \ref{Cowen-Douglas}. 

(ii)$\Rightarrow$(i) It suffices to check that $T$ is analytic. By Oka-Grauert's theorem (see \cite[P.~71, Corollary~2.17]{L1990}, \cite[P.~3]{ES2014}), every holomorphic vector bundle on a bidisc
is holomorphically trivial. 
Combining this with the proof of \cite[Theorem~4.5]{ES2014} shows that if 
$T^* \in {\bf B}_1(\mathbb D^2_r),$ then $T$ is unitarily equivalent to the multiplication $2$-tuple $\mathscr M_z$ on a reproducing kernel Hilbert space of scalar-valued holomorphic functions on $\mathbb D^2_r.$ Since $\mathscr M_z$ is analytic, $T$ is analytic.  
\end{proof}

The conclusion of Theorem~\ref{model-thm} can be rephrased as follows:
\begin{corollary} \label{model-coro}
A cyclic analytic toral $2$-isometric $2$-tuple on $\mathcal H$ is unitarily equivalent to the multiplication pair $\mathscr M_z$ on $\mathcal D(\mu_1, \mu_2)$ if and only if $\ker T^*$ is 
wandering subspace for $T$ spanned by a cyclic vector for $T.$
\end{corollary}
\begin{remark} Let $\mathcal D$ denote the Dirichlet space (that is, the Dirichlet-type space associated with the Lebesgue measure on the unit circle) and let $\mathscr M_w$ be the operator of multiplication by $w$ on $\mathcal D$. It is easy to see that the commuting pair $T=(\mathscr M_w, \mathscr M_w)$ is a cyclic analytic toral $2$-isometry on $\mathcal D.$ Note that $\ker T^*=\ker D^*$ is spanned by $1$ and it is not a wandering subspace for $T.$ It is evident that $T$ is not unitarily equivalent to the multiplication pair $\mathscr M_z$ on $\mathcal D(\mu_1, \mu_2)$ for any $\mu_1, \mu_2 \in M_+(\mathbb T).$  
\end{remark}

The following is a $2$-variable analog of \cite[Theorem~5.2]{R1991}.
\begin{proposition} \label{unitary-measure}
For $j=1, 2,$ let $\mu^{(j)}_1, \mu^{(j)}_2 \in M_+(\mathbb T).$ Then the multiplication $2$-tuple $\mathscr M^{(1)}_{z}$ on $\mathcal D(\mu^{(1)}_1, \mu^{(1)}_2)$ is unitarily equivalent to the multiplication $2$-tuple $\mathscr M^{(2)}_z$ on $\mathcal D(\mu^{(2)}_1, \mu^{(2)}_2)$ if and only if $\mu^{(1)}_j = \mu^{(2)}_j,$ $j=1, 2.$ 
\end{proposition}
\begin{proof} 
Suppose there is a unitary operator $U : \mathcal D(\mu^{(1)}_1, \mu^{(1)}_2) \rar \mathcal D(\mu^{(2)}_1, \mu^{(2)}_2)$ such that 
\beq \label{intertwining-r}
\mathscr M^{(2)}_{z_j} U = U \mathscr M^{(1)}_{z_j}, \quad j=1, 2. 
\eeq
Since the joint kernel of the adjoint of multiplication tuples is spanned by $1,$ by \eqref{intertwining-r}, $U$ must map $1$ to some constant of modulus $1.$ After multiplying $U$ by a unimodular constant, if required, we may assume that $U1=1.$ It now follows from \eqref{intertwining-r} that $U$ is identity on polynomials. By Lemma~\ref{Richter-formula} (applied twice), we obtain for any polynomial $p$ in two variables,
\beqn
\int_{\mathbb T^2} |p(e^{i \eta},  e^{i \theta})|^2 d\mu^{(1)}_1(\eta)d\theta 
&=& \int_{\mathbb T^2} |p(e^{i \eta},  e^{i \theta})|^2 d\mu^{(2)}_1(\eta)d\theta, \\
\int_{\mathbb T^2} |p(e^{i \theta}, e^{i \eta})|^2 d\mu^{(1)}_2(\eta)d\theta 
&=& \int_{\mathbb T^2} |p(e^{i \theta}, e^{i \eta})|^2 d\mu^{(2)}_2(\eta)d\theta. 
\eeqn
It is easy to see that for any polynomial $p$ in one variable, 
\beqn
\int_{\mathbb T} |p(e^{i \eta})|^2 d\mu^{(1)}_j(\eta) 
= \int_{\mathbb T} |p(e^{i \eta})|^2 d\mu^{(2)}_j(\eta), \quad j=1, 2. 
\eeqn
Combining polarization identity with the uniqueness of the trigonometric moment problem yields the desired uniqueness. 
\end{proof}
\begin{remark} 
One may use Lemma~\ref{Richter-formula} and argue as in \cite[Theorem~6.2]{R1991} to obtain the following fact:
{\it For $j=1, 2,$ let $\mu^{(j)}_1, \mu^{(j)}_2 \in M_+(\mathbb T).$ Then $$\mathcal D(\mu^{(1)}_1, \mu^{(1)}_2) \subseteq \mathcal D(\mu^{(2)}_1, \mu^{(2)}_2)$$ if and only if $\mu^{(2)}_j  \ll \mu^{(1)}_j$ and the Radon-Nikod\'{y}m derivative ${d\mu^{(2)}_j}/{d\mu^{(1)}_j} \in L^\infty(\mathbb T),$ $j=1, 2.$} 
We leave the details to the reader. 
\end{remark}

We conclude this section with an application to toral isometries. 
\begin{corollary} \label{toral-iso-coro}
Let $T$ be a cyclic analytic toral isometry with cyclic vector $f_0 \in \ker T^*.$ Then the following statements are equivalent:
\begin{enumerate}[$\rm(i)$]
\item $\ker T^*$ is a wandering subspace for $T,$
\item $T$ is unitarily equivalent to $\mathscr M_z$ on $H^2(\mathbb D^2),$
\item $T$ is doubly commuting, that is, $T^*_{j}T_{i}=T_{i}T^*_{j},$ $1 \Le i \neq j \Le 2.$
\end{enumerate}
\end{corollary}
\begin{proof} (i)$\Rightarrow$(ii) By Theorem~\ref{model-thm}, there exist $\mu_1, \mu_2 \in M_+(\mathbb T)$ such that $T$ is unitarily equivalent to $\mathscr M_z$ on $\mathcal D(\mu_1, \mu_2).$ Since $T$ is a toral isometry, $\mathscr M_z$ is also a toral isometry. It now follows from 
\eqref{formula-Richter} that for every polynomial $p$ in two variables, 
\beqn
\int_{\mathbb T^2} |p(e^{i \eta},  e^{i \theta})|^2 d\mu_1(\eta)d\theta =0, \quad \int_{\mathbb T^2} |p(e^{i \theta}, e^{i \eta})|^2 d\mu_2(\eta)d\theta =0. 
\eeqn
One may now argue as the proof of Proposition~\ref{unitary-measure} to conclude that $\mu_1=0$ and $\mu_2=0.$ This yields (ii).

The implication (ii)$\Rightarrow$(iii) is a routine verification, while the implication  
(iii)$\Rightarrow$(i) is recorded in Remark~\ref{rmk-wandering}. 
\end{proof}

\section{Concluding remarks}

We conclude the paper with a brief discussion on the spectral picture of the multiplication $2$-tuple $\mathscr M_z$ on $\mathcal D(\mu_1, \mu_2).$ 
We claim that 
\beq
\label{Taylor-sp}
\sigma(\mathscr M_z) &=& \overline{\mathbb D}^2, \\
\label{Taylor-ess-sp}
\sigma_e(\mathscr M_z) & \subseteq & \overline{\mathbb D}^2 \setminus \Omega
\eeq
for some open set $\Omega$ in $\mathbb C^2$ containing $\big(\mathbb D \times \{0\}\big) \cup \big(\{0\} \times \mathbb D\big).$
To see \eqref{Taylor-sp}, note that by \cite[Theorem~4.9]{Cu1988}, 
for any commuting pair $T=(T_1, T_2),$ $\sigma(T) \subseteq \sigma(T_1) \times \sigma(T_2).$ Since the spectrum of any $2$-isometry is contained in $\overline{\mathbb D}$ (see \cite[Lemma~1.21]{AS1995}) and both $\mathscr M_{z_1}$ and $\mathscr M_{z_2}$ are $2$-isometries (see Corollary~\ref{cyclic-coro}), we obtain
$
\sigma(\mathscr M_z) \subseteq \overline{\mathbb D}^2.$
Also, by Corollary~\ref{T-spectrum}, $\mathbb D^2 \subseteq \sigma_p(\mathscr M^*_z) \subseteq \sigma(\mathscr M^*_z).$ Since $\sigma(\mathscr M^*_z)=\{\overline{z} : z \in \sigma(\mathscr M_z)\},$ we have the inclusion $\mathbb D^2 \subseteq  \sigma(\mathscr M_z).$ Finally, since the Taylor spectrum is closed (see \cite[Corollary~4.2]{Cu1988}), we obtain \eqref{Taylor-sp}. 
On the other hand, 
an examination of the proof of Theorem~\ref{Gleason-new} (using the full power of Lemma~\ref{division-slice} together with Lemma~\ref{contractive-h}) shows that  
\beqn
\big(\mathbb D \times \{0\}\big) \cup \big(\{0\} \times \mathbb D\big) \subseteq  \mathbb C^2 \setminus \sigma_e(\mathscr M_z).
\eeqn
Since the essential spectrum is a closed subset of the Taylor spectrum, \eqref{Taylor-ess-sp} now follows from \eqref{Taylor-sp}.
The natural question arises {\it whether the unit bidisc lies in the complement of the essential spectrum of $\mathscr M_z$} (there are interesting examples of toral $2$-isometries supporting this possibility; see \cite[Proposition~5(iii)]{AS1999}). If this question has an affirmative answer, then $\sigma_e(\mathscr M_z) = \partial (\mathbb D^2).$
Indeed,  if $\lambda \in \partial(\mathbb D^2) \backslash \sigma_e(\mathscr M_z),$ then there exist two sequences in $\mathbb D^2$ and $\mathbb C^2 \backslash \overline{\mathbb D}^2$ converging to $\lambda,$ which together with the continuity of the Fredholm index (see \eqref{index})
leads to a contradiction. This in turn leads to an improvement of  Corollary~\ref{Cowen-Douglas} providing a bidisc analog of \cite[Corollary~3.8]{R1991} and also solves Gleason's problem for $\mathcal D(\mu_1, \mu_2)$ (see Lemma~\ref{abstract-Gleason-p}).  

\medskip \textit{Acknowledgement}. \ The authors would like to thank Archana Morye, Shibananda Biswas and Somnath Hazra for some fruitful conversations on the subject of this article.

%

\end{document}